\documentclass[a4paper,english,12pt]{article}

\usepackage{amsmath,amsfonts,amssymb,amsthm,bbm,graphics,epsfig,psfrag,epstopdf,dsfont,esint,mathtools,yhmath}
\usepackage{geometry}
\usepackage[active]{srcltx}
\usepackage{paralist,subfigure,float}
\usepackage{color,soul}
\usepackage[dvipsnames]{xcolor}
\usepackage{hyperref} 

\setstcolor{red}
\newtheorem{theorem}{Theorem}[section]
\newtheorem{corollary}[theorem]{Corollary}

\newtheorem{proposition}[theorem]{Proposition}

\newcommand{\R}{\mathbb{R}}
\def\N{\mathbb{N}}
\def\Z{\mathbb{Z}}

\def\epsilon{\varepsilon}

\def\tilde{\widetilde}
\DeclareMathOperator\diam{diam}
\DeclareMathOperator{\dist}{dist}

\definecolor{aquamarine}{rgb}{0.13, 0.68, 0.8}

\newcommand{\be}{\begin{equation}}
\newcommand{\ee}{\end{equation}}
\newcommand{\baa}{\begin{array}}
\newcommand{\eaa}{\end{array}}
\newcommand{\ba}{\begin{eqnarray}}
\newcommand{\ea}{\end{eqnarray}}

\numberwithin{equation}{section}

\begin{document}
\date{}
\title{\bf{Propagation or extinction in bistable equations: the non-monotone role of initial fragmentation}}
\author{Matthieu Alfaro$^{\hbox{\small{ a}}}$, Fran\c cois Hamel$^{\hbox{\small{ b}}}$ and Lionel Roques$^{\hbox{\small{ c }}}$\thanks{This work has received funding from Excellence Initiative of Aix-Marseille Universit\'e~-~A*MIDEX, a French ``Investissements d'Avenir'' programme, from the French ANR RESISTE (ANR-18-CE45-0019), ReaCh (ANR-23-CE40-0023-02) and DEEV (ANR-20-CE40-0011-01) projects.}\\
\\
\footnotesize{$^{\hbox{a }}$Univ Rouen Normandie, LMRS, CNRS, Rouen, France}\\
\footnotesize{$^{\hbox{b }}$Aix Marseille Univ, CNRS, I2M, Marseille, France} \\
\footnotesize{$^{\hbox{c }}$INRAE, BioSP, 84914, Avignon, France}\\
}
\maketitle

\begin{center}
{\it To Professor Yihong Du, a distinguished scholar and esteemed mathematician}
\end{center}
	
\begin{abstract}
\noindent{}In this paper, we investigate the large-time behavior of bounded solutions of the Cauchy problem for a reaction-diffusion equation in $\R^N$ with bistable reaction term. We consider initial conditions that are chiefly indicator functions of bounded Borel sets. We examine how geometric transformations of the supports of these initial conditions affect the propagation or extinction of the solutions at large time. We also consider two fragmentation indices defined in the set of bounded Borel sets and we establish some propagation or extinction results when the initial supports are weakly or highly fragmented. Lastly, we show that the large-time dynamics of the solutions is not monotone with respect to the considered fragmentation indices, even for equimeasurable sets.
\end{abstract}

%%%%%%%%%%%%%%%%%%%%%%%%%%%%%%%%%%%%%%%%%%%%%%%%%%%%%%%%%%
%%%%%%%%%%%%%%%%%%%%%%%%%%%%%%%%%%%%%%%%%%%%%%%%%%%%%%%%%%

\section{Introduction}\label{intro}

This paper is concerned with the large-time behavior of solutions to the Cauchy problem for the bistable reaction-diffusion equation
\be\label{homo}\left\{\baa{ll}
\displaystyle\frac{\partial u}{\partial t}=\Delta u+f(u), & t>0,\ x\in\R^N,\vspace{3pt}\\
u(0,x)=u_0(x), & x\in\R^N,\eaa\right.
\ee
in any dimension $N\ge1$, where $\Delta$ stands for the Laplacian with respect to the spatial variables $x\in\R^N$. The function $f:[0,1]\to\R$ is assumed to be of class $C^1$ and of the bistable type with positive mass, that~is,
\be\label{bistable}\left\{\baa{ll}
\displaystyle f(0)=f(1)=0,\ \ f'(0)<0,\ \ f'(1)<0,\ \ \int_0^1f(s)ds>0,\vspace{3pt}\\
\exists\,\theta\in(0,1),\ \ f<0\hbox{ in }(0,\theta),\ f>0\hbox{ in }(\theta,1),\ f'(\theta)>0.\eaa\right.
\ee

The initial conditions $u_0:\R\to[0,1]$ are assumed to be Lebesgue-measurable and compactly supported. The Cauchy problem~\eqref{homo} is well posed and, for each such $u_0$, the solution $u:[0,+\infty)\times\R^N\to[0,1]$ exists and is unique, it is classical in $(0,+\infty)\times\R^N$ and $u(t,\cdot)\to u_0$ as $t\to0^+$ in $L^1(\R^N)$. Furthermore, from the strong parabolic maximum principle, one has $0<u(t,x)<1$ for all $(t,x)\in(0,+\infty)\times\R^N$ provided that $\|u_0\|_{L^1(\R^N)}>0$. In biological or ecological models, the quantity $u$ stands for the normalized concentration of a species, subject to local dispersion on the one hand, and on growth and death processes taking into account a strong Allee effect on the other hand, meaning that the per capita growth rate $f(u)/u$ (hence $f(u)$ itself) is negative at low densities.

\subsubsection*{Some notations}

Throughout the paper, ``$\|\ \|_{\infty}$", ``$|\ |$'' and ``$\ \cdot\ $'' denote respectively the maximum norm, the Euclidean norm and the Euclidean inner product in $\R^N$, $B_r(x):=\{y\in\R^N:|y-x|<r\}$ denotes the open Euclidean ball of center $x\in\R^N$ and radius $r\ge0$, and $B_r:=B_r(0)$. We call
$$\mathcal{B}:=\big\{B_r(x):x\in\R^N,\ r>0\big\}.$$
We denote $\lambda$ the Lebesgue measure in $\R^N$, and $\omega_N:=\lambda(B_1)$. For $x\in\R^N$ and a Borel subset $E$ of $\R^N$, we define the essential distance between $x$ and $E$ as
$$\dist(x,E):=\mathrm{essinf}\big(y\in E\mapsto|x-y|\big)=\sup\big\{r\ge0:\lambda(B_r(x)\cap E)=0\big\}$$
(hence $\dist(x,E)=+\infty$ if $\lambda(E)=0$, and the supremum is a maximum if $\lambda(E)>0$). The function $x\mapsto\mathrm{dist}(x,E)$ is then Lipschitz continuous in $\R^N$. The~$d_1$-distance between two bounded Borel subsets $E$ and $F$ of $\R^N$ is given by the Lebesgue measure of the symmetric difference~$E\Delta F$, that is,
\be\label{defd1}
d_1(E,F):=\lambda\big((E\cup F)\setminus(E\cap F)\big)=\lambda(E\cup F)-\lambda(E\cap F),
\ee
which is also the $L^1$ norm of the difference between the indicator functions of the sets $E$ and $F$. The essential Hausdorff distance between two bounded Borel subsets $E$ and $F$ of~$\R^N$ is given by
\be\label{defdH}
d_{H}(E,F):=\max\Big(\|\dist(\cdot,F)\|_{L^\infty(E)},\|\dist(\cdot,E)\|_{L^\infty(F)}\Big),
\ee
with the conventions $d_{H}(E,F)=d_{H}(F,E)=+\infty$ if $\lambda(E)>0=\lambda(F)$, and $d_{H}(E,F)=0$ if $\lambda(E)=\lambda(F)=0$. Throughout the paper, for two Borel subsets $E$ and $F$ of $\R^N$, we also say that $E$ is included in $F$ up to a negligible set, or equivalently $F$ contains $E$ up to a negligible set, if $\lambda(E\setminus F)=0$. We say that $E$ and $F$ are equal up to a negligible set if $E\subset F$ up to a negligible set and $F\subset E$ up to a negligible set, that is, the indicator functions of the sets~$E$ and $F$ are equal almost everywhere in $\R^N$. In other words, for any two bounded Borel subsets $E$ and $F$ of $\R^N$, $d_1(E,F)=0$ if and only if $E$ and $F$ are equal up to a negligible set. Lastly, for a bounded Borel set $E\subset\R^N$ and for $r>0$, we call
\be\label{defB1}
\mathbb{B}_r(E):=\big\{F\hbox{ bounded Borel subsets of $\R^N$}: d_1(E,F)<r\big\}.
\ee
%and
%\be\label{defBH}
%\mathbb{B}^H_r(E):=\big\{F\hbox{ bounded Borel subsets of $\R^N$}: d_H(E,F)<r\big\}.
%\ee

\subsubsection*{Some results of the literature}

The main goal of the paper is to understand the role of the fragmentation, in a sense to be made precise, of the initial condition $u_0$ on the large-time dynamics of the solution~$u$. But before presenting various notions of fragmentation and their properties and effects on the dynamics of~\eqref{homo}, let us recall some important known results of the literature.

First of all, it turns out that, for each given compactly supported initial condition $u_0:\R^N\to[0,1]$, $u$ can have only three types of asymptotic behaviors as $t\to+\infty$. Namely, it follows from~\cite{p1} that either $u(t,\cdot)\to0$ as $t\to+\infty$ uniformly in $\R^N$ (the extinction case), or $u(t,\cdot)\to 1$ as $t\to+\infty$ locally uniformly in $\R^N$ (the invasion case), or $u(t,\cdot)\to\Phi(\cdot+x_0)$ as $t\to+\infty$ uniformly in $\R^N$ for some $x_0\in\R^N$, where $\Phi:\R^N\to(0,1)$ is the unique stationary solution of~\eqref{homo} converging to $0$ at infinity and such that $\max_{\R^N}\Phi=\Phi(0)$ (the function~$\Phi$ is actually radially symmetric and decreasing with respect to the origin). We point out that, if $f$ satisfies~\eqref{bistable} but with a nonpositive integral over $[0,1]$, then only the extinction case is possible. Under the full assumption~\eqref{bistable}, when $u(t,\cdot)\to1$ as $t\to+\infty$ locally uniformly in $\R^N$, then it is also known from~\cite{fm} (in dimension $N=1$) and from~\cite{u2} (in any dimension $N\ge1$) that
$$\liminf_{t\to+\infty}\Big(\min_{|x|\le ct-\frac{N-1}{c}\ln t-A}u(t,x)\Big)\mathop{\longrightarrow}_{A\to+\infty}1\ \hbox{ and }\ \limsup_{t\to+\infty}\Big(\max_{|x|\ge ct-\frac{N-1}{c}\ln t+A}u(t,x)\Big)\mathop{\longrightarrow}_{A\to+\infty}0,$$
where $c>0$ is the unique speed of a traveling front $\varphi(x-ct)$ solving~\eqref{homo} in dimension~$N=1$ and connecting $1$ to $0$, that is, with $\varphi:\R\to(0,1)$ such that~$\varphi(-\infty)=1$ and~$\varphi(+\infty)=0$. In other words, the levels sets of $u$ with a given level $\rho\in(0,1)$ spread with speed $c$ in all directions as $t\to+\infty$ (this spreading result was originally proved in~\cite{aw}) and are even located at bounded Hausdorff distance from spheres of radii $ct-((N-1)/c)\ln t$. As a matter of fact, the level sets become asymptotically locally planar as $t\to+\infty$, since~$\sup_{t\ge A,\,|x|\ge A}\nabla u(t,x)/|\nabla u(t,x)|+x/|x|\to0$ as $A\to+\infty$ by~\cite{j} (see~\cite{hr2,hr3} for a study of further flattening properties), but they nevertheless do not necessarily converge to families of spheres, see~\cite{rr3,r1,r3}. For further spreading properties and estimates of the location of the level sets for various types of functions $f$, we refer to e.g.~\cite{dm2,d,g,rrr} in the case of compactly supported initial conditions, and to~\cite{b,f,hnrr,kpp,l,nrr,p4,u1} for initial conditions with unbounded initial support in dimension $N=1$ and~\cite{hr1,hr2,hr3,mn,mnt,p3,rr1,rr2} in any dimension.

Moreover, if for the problem~\eqref{homo}-\eqref{bistable} one considers a family $[0,+\infty)\ni\sigma\mapsto u^\sigma_0$ of compactly supported initial conditions, which is continuous and increasing in the $L^1(\R^N)$ sense and which is such that $u^0_0=0$, there is a unique threshold $\sigma^*\in(0,+\infty]$ such that the solutions $u^\sigma$ of~\eqref{homo} with initial conditions $u^\sigma_0$ satisfy: 1)~$u^\sigma(t,\cdot)\to0$ as $t\to+\infty$ uniformly in $\R^N$ if $0\le\sigma<\sigma^*$, 2)~$u^\sigma(t,\cdot)\to1$ as $t\to+\infty$ locally uniformly in $\R^N$ if~$\sigma^*<+\infty$ and $\sigma>\sigma^*$, 3)~$u^{\sigma^*}(t,\cdot)\to\Phi(\cdot+x_0)$ as $t\to+\infty$ uniformly in $\R^N$ for some~$x_0\in\R^N$ if $\sigma^*<+\infty$, see~\cite{p1}. In particular, for any given $\alpha\in(\theta,1]$, since the solutions $u^{\alpha,R}$ of~\eqref{homo} with initial conditions
$$u_0^{\alpha,R}=\alpha\,\mathds{1}_{B_R}:=\left\{\baa{ll}\alpha &\hbox{if }|x|<R,\vspace{3pt}\\ 0&\hbox{if }|x|\ge R,\eaa\right.$$
converge to $1$ as $t\to+\infty$ locally uniformly in $\R^N$ provided $R>0$ is large enough (see e.g.~\cite{aw,dp}), there is then from~\cite{mz2} or~\cite{p1} a unique $R_{\alpha}\in(0,+\infty)$ such that 
\be\label{defRalpha}\left\{\baa{ll}
u^{\alpha,R}\to0\hbox{ as $t\to+\infty$ uniformly in $\R^N$} & \hbox{if $0\le R<R_{\alpha}$},\vspace{3pt}\\
u^{\alpha,R}\to1\hbox{ as $t\to+\infty$ locally uniformly in $\R^N$} & \hbox{if $R>R_{\alpha}$},\vspace{3pt}\\
u^{\alpha,R}\to\Phi\hbox{ as $t\to+\infty$ uniformly in $\R^N$} & \hbox{if $R=R_{\alpha}$}.\eaa\right.
\ee
The asymptotic behavior of~$R_{\alpha}$ as $\alpha\displaystyle{\mathop{\to}^>}\theta$ was estimated in~\cite{adf}. The first sharp threshold results of that type were obtained in dimension $N=1$, in~\cite{z} with initial conditions that are indicator functions of bounded intervals, and then in~\cite{dm1} for more general initial conditions. In both papers~\cite{dm1,z}, more general functions $f$ can be considered (such as ignition nonlinearities for which there is $\theta\in(0,1)$ such that $f=0$ in $[0,\theta]\cup\{1\}$, $f>0$ in $(\theta,1)$ and $f$ is non-decreasing in a neighborhood of $\theta$). We point out that reference~\cite{dm1} also contains a general convergence result for the bounded nonnegative solutions of~\eqref{homo} with compactly supported initial conditions in dimension $N=1$ under the sole assumption $f(0)=0$: namely such solutions converge as $t\to+\infty$ locally in~$\R$ to a stationary solution, which is either constant or even and decreasing with respect to a point. Other sharp threshold results were obtained in~\cite{mp1,p1} for bistable-type autonomous or non-autonomous equations in~$\R$ or~$\R^N$ with initial conditions converging to $0$ at infinity, and in~\cite{mz1,mz2} for more general functions~$f$ and radially symmetric non-increasing initial conditions in $L^2(\R)$ and $L^2(\R^N)$. Earlier references~\cite{aw,k,lk} also established non-sharp extinction/invasion results with respect to the size or the amplitude of the initial condition~$u_0$ in $\R^N$ for various reaction terms $f$. We refer to~\cite{dm1,dp,mp1,mp2,p2} for further results on the convergence to a stationary solution or the convergence to the set of stationary solutions (quasiconvergence) for various equations of the type~\eqref{homo}. Extinction and invasion results have also been derived for equations of the type~\eqref{homo} in general unbounded domains~$\Omega$, see~\cite{bhn,r2}. We lastly mention~\cite{ahk,mntm,ntm} some results on the optimization of $\int_\Omega u(T,\cdot)$ (or other integral quantities) with respect to the initial condition under some pointwise and integral constraints, for some reaction-diffusion equations set in bounded domains $\Omega$ with Neumann boundary conditions on $\partial\Omega$.

\subsubsection*{Main goal and outline of the paper}

In this paper, we consider the Cauchy problem~\eqref{homo} with a bistable function $f$ of the type~\eqref{bistable}, and with initial conditions which are mostly indicator functions of bounded Borel sets $E$, that is,
\be\label{defu0}
u_0(x)=\mathds{1}_E(x):=\left\{\baa{ll}1&\hbox{if }x\in E,\vspace{3pt}\\ 0&\hbox{if }x\in\R^N\!\setminus\!E.\eaa\right.
\ee
Our goal is to understand the effect of the fragmentation of the initial set $E$ on the large-time dynamics of the solution of~\eqref{homo} with initial condition~\eqref{defu0}. Unlike the aforementioned sharp threshold results obtained for monotone families of initial conditions, that would correspond in~\eqref{defu0} to monotone (with respect to the inclusion) families of sets $E$, we typically consider here some Borel sets $E$ having the same Lebesgue measure (hence being in general not comparable with respect to the inclusion) and we look for some properties of these sets which guarantee the extinction or the invasion of the solutions~$u$. As we will see, the large-time dynamics will strongly depend on the fragmentation of the sets $E$. Two fragmentation indices are considered and their properties are discussed in Section~\ref{sec2}. The main results on the role of the fragmentation of $E$ on the large-time dynamics of the solutions of~\eqref{homo} and~\eqref{defu0} are stated in Section~\ref{sec3}. We especially show that there is no monotonicity of the large-time dynamics with respect to the fragmentation indices in the class of equimeasurable sets. The main proofs are given in Section~\ref{sec4}.

%%%%%%%%%%%%%%%%%%%%%%%%%%%%%%%%%%%%%%%%%%%%%%%%%%%%%%%%%%
%%%%%%%%%%%%%%%%%%%%%%%%%%%%%%%%%%%%%%%%%%%%%%%%%%%%%%%%%%

\section{Fragmentation indices and their properties}\label{sec2}

For the solutions $u$ of problem~\eqref{homo}-\eqref{bistable} with initial conditions of the type~\eqref{defu0} with $E=B_R$, there is a unique $R_{1}\in(0,+\infty)$ such that extinction holds as $t\to+\infty$ if~$R<R_{1}$ while invasion happens if $R>R_{1}$, as recalled in Section~\ref{intro}, see~\eqref{defRalpha}. Now, for a set~$E$ with Lebesgue measure $\lambda(E)$ larger than $\lambda(B_{R_{1}})$, can one provide some conditions guaranteeing the extinction or the invasion? The invasion obviously holds if $E$ itself is a ball or if $E$ contains a ball of radius larger than $R_{1}$, from the maximum principle. But the invasion does not hold in general. For instance, in dimension $N=1$, it follows from~\cite[Theorem~2]{grh} that, if
$$D_a:=(-a-r,-a)\cup(a,a+r)$$
with a given $r\in(0,2R_{1})$, then the solutions $u^a$ of~\eqref{homo} with initial conditions $\mathds{1}_{D_a}$ go to extinction as $t\to+\infty$ for all $a>0$ large enough, while~$\lambda(D_a)=2r>2R_{1}=\lambda(B_{R_{1}})$ if $r$ is chosen such that $r>R_{1}$. 

In order to quantify how a bounded non-negligible Borel set $E$ deviates from the set of balls having the same measure as $E$, we consider in this paper two different fragmentation indices. We then list their main properties and compare them. For such a set $E$, we define
$$R_E:=\Big(\frac{\lambda(E)}{\omega_N}\Big)^{1/N}=\Big(\frac{\lambda(E)}{\lambda(B_1)}\Big)^{1/N}$$
which is nothing but the radius of all equimeasurable balls. We then denote $\rho_E$ the smallest radius of a ball containing~$E$ up to a negligible set, which can also be expressed as
\be\label{defrhoE}
\rho_E=\min_{x\in\R^N}\|d(x,\cdot)\|_{L^\infty(E)}.
\ee
Notice that
\be\label{rhodiam}
\rho_E\ge\frac{\diam(E)}{2}
\ee
for any $N\ge1$, with equality when $N=1$, where
$$\diam(E):=\|\,(x,y)\mapsto|x-y|\,\|_{L^\infty(E\times E)}$$
is the essential diameter of $E$. There is actually a unique $x_E$ such that $E\subset B_{\rho_E}(x_E)$ up to a negligible set. Lastly,
$$R_E\le\rho_E$$
and the equality holds if and only if $E$ is a ball, up to a negligible set.

\subsubsection*{The fragmentation index $\delta_1$}

The first considered fragmentation index of a bounded non-negligible Borel set $E\subset\R^N$ is, up to the multiplicative factor $1/2$, the Fraenkel asymmetry. It is based on the $L^1$ distance between $E$ and the equimeasurable balls:
$$\delta_1(E):=\inf_{B\in\mathcal{B},\,\lambda(B)=\lambda(E)}\frac{\|\mathds{1}_B-\mathds{1}_E\|_{L^1(\R^N)}}{2\lambda(E)}.$$
Since $\|\mathds{1}_B-\mathds{1}_E\|_{L^1(\R^N)}=\lambda(B\Delta E)=2\lambda(E)-2\lambda(E\cap B_{R_E}(x))=2\lambda(E\!\setminus\!B_{R_E}(x))$ for every ball $B=B_{R_E}(x)$ (namely, such that $\lambda(B)=\lambda(E)$), it follows from the continuity of the maps $x\mapsto\lambda(E\!\setminus\!B_{R_E}(x))$ and $x\mapsto\lambda(E\cap B_{R_E}(x))$ in~$\R^N$ that
\be\label{defdelta1}\baa{rcl}
\delta_1(E)=\displaystyle\min_{x\in\R^N}\frac{\lambda(E\!\setminus\!B_{R_E}(x))}{\lambda(E)} & = & \displaystyle\min_{x\in\R^N}\frac{\lambda(E)-\lambda(E\cap B_{R_E}(x))}{\lambda(E)}\vspace{3pt}\\
& = & \displaystyle1-\max_{x\in\R^N}\frac{\lambda(E\cap B_{R_E}(x))}{\lambda(E)}.\eaa
\ee
In particular,
$$0\le\delta_1(E)<1,$$
and $\delta(E)=0$ if and only if $E$ is a ball, up to a negligible set. Furthermore, the constant~$1$ in the inequality $\delta_1(E)<1$ is optimal: for instance, the non-negligible Borel sets
\be\label{defEn}
E_n:=\bigcup_{x\in\Z^N\cap(0,n)^N}B_{1/n^2}\Big(\frac{x}{n}\Big),\ \ n\ge 1,
\ee
which are all included into the cube $(0,1)^N$, are highly fragmented for the index $\delta_1$ as $n\to+\infty$, in the sense that
\be\label{delta1En}
\delta_1(E_n)\to1\ \hbox{ as $n\to+\infty$}.
\ee
Lastly, for any $\nu\in(0,1)$, consider the non-negligible Borel sets
\be\label{defFn}
F_n:=\bigcup_{x\in\Z^N\cap B_n}\Big[\frac{x}{n}+\Big(\!\!-\!\frac{\nu}{2n},\frac{\nu}{2n}\Big)^N\Big]=\bigcup_{y\in (\Z^N\!/n)\cap B_1}\Big[y+\Big(\!\!-\!\frac{\nu}{2n},\frac{\nu}{2n}\Big)^N\Big],\ \ n\ge 1,
\ee
which are all included into the ball $B_{1+\nu\sqrt{N}/2}$. Actually, counting the number $A_n$ of points $x\in\Z^N\cap B_n$ corresponds to the $N-$dimensional extension of Gauss circle problem. It is a difficult issue, but some estimates can be obtained when $n$ is large. In particular, $A_n=n^N\, \omega_N + O(n^{N-1})$ as $n\to+\infty$, see \cite[Chapter~4]{Kra89} for sharper estimates. Thus, $\lambda(F_n)=A_n\,(\nu/n)^N\to\nu^N\, \omega_N $ as $n\to +\infty$. The same arguments imply that $\lambda(F_n \cap B_\nu)\to\nu^{2N}\,\omega_N$ as $n\to +\infty$, so that
$$\delta_1(F_n)\to1-\nu^N\ \hbox{ as $n\to+\infty$}.$$

\subsubsection*{The fragmentation index $\delta_H$}

The second considered fragmentation index of a bounded non-negligible Borel set~$E\subset\R^N$ is based on the essential Hausdorff distance between~$E$ and the equimeasurable balls:
$$\delta_H(E):=\inf_{B\in\mathcal{B},\,\lambda(B)=\lambda(E)}\frac{d_{H}(E,B)}{\rho_E+R_E}.$$
From the continuity of the map $x\mapsto d_{H}(E,B_{R_E}(x))$ in $\R^N$, one can also write
\be\label{defdeltaH}
\delta_H(E)=\min_{x\in\R^N}\frac{d_{H}(E,B_{R_E}(x))}{\rho_E+R_E}.
\ee
Furthermore, for the unique $x_E\in\R^N$ such that $E\subset B_{\rho_E}(x_E)$ up to a negligible set, one has $\|\dist(\cdot,B_{R_E}(x_E))\|_{L^\infty(E)}\le\rho_E-R_E<\rho_E+R_E$. Moreover, $\|\dist(\cdot,E)\|_{L^\infty(B_{R_E}(x_E))}<\rho_E+R_E$ (indeed, otherwise, since the map $\dist(\cdot,E)$ is continuous in $\R^N$, there would exist a point $y_E\in\overline{B_{R_E}(x_E)}$ such that $\dist(y_E,E)\ge\rho_E+R_E$, hence $E\subset\R^N\setminus B_{\rho_E+R_E}(y_E)$ up to a negligible set, and then $E\subset\R^N\setminus B_{\rho_E}(x_E)$ up to a negligible set from the triangle inequality, a contradiction with the positivity of~$\lambda(E)$ and the inclusion $E\subset B_{\rho_E}(x_E)$ up to a negligible set). Finally, the inequality $\|\dist(\cdot,E)\|_{L^\infty(B_{R_E}(x_E))}<\rho_E+R_E$ has been proved, hence $d_{H}(E,B_{R_E}(x_E))<\rho_E+R_E$ and
$$0\le\delta_H(E)<1.$$
As for the index $\delta_1$, a bounded non-negligible Borel $E$ satisfies $\delta_{H}(E)=0$ if and only if $E$ is a ball, up to a negligible set. Furthermore, the constant $1$ in the inequality $\delta_{H}(E)<1$ is optimal: for instance, the sets $E_n$ defined in~\eqref{defEn} satisfy $\delta_H(E_n)\to1$ as $n\to+\infty$. Similarly, taking two points $x\neq y$ in $\R^N$, the non-negligible Borel sets
\be\label{defGn}
G_n:=B_{1/n}(x)\cup B_{1/n}(y),\ \ n\ge1,
\ee
which are all included in the ball $B_{|x-y|/2+1}((x+y)/2)$, are highly fragmented for the index~$\delta_H$ as $n\to+\infty$, in the sense that
$$\delta_H(G_n)\to1\ \hbox{ as $n\to+\infty$}.$$
Observe on the other hand that these sets are not highly fragmented for the index~$\delta_1$, since $\delta_1(G_n)\to1/2$ as $n\to+\infty$. Lastly, for $\nu\in(0,1)$, the sets $F_n$ defined by~\eqref{defFn} are such that
$$\delta_H(F_n)\to\frac{1-\nu}{1+\nu}\ \hbox{ as $n\to+\infty$}.$$

In dimension $N=1$, for a bounded non-negligible Borel set $E$, calling
$$m_E:=\mathrm{essinf}\,E\ \hbox{ and }M_E:=\mathrm{esssup}\,E,$$
one has $m_E<M_E\in\R$,  $x_E=(m_E+M_E)/2$, and $\diam(E)=M_E-m_E=2\rho_E\ge2R_E$. Moreover, $\|\dist(\cdot,(x-R_E,x+R_E))\|_{L^\infty(E)}\ge\rho_E-R_E$ for every $x\in\R$, with equa\-lity if and only if $x=x_E$, while $\|\dist(\cdot,E)\|_{L^\infty(x_E-R_E,x_E+R_E)}\le\rho_E-R_E$ (otherwise, there would be $y_E\in[x_E-R_E,x_E+R_E]$ such that $d:=\dist(y_E,E)>\rho_E-R_E$ and $E\subset(x_E-\rho_E,x_E+\rho_E)\setminus(y_E-d,y_E+d)$ up to a negligible set, a contradiction with~$|y_E-x_E|\le R_E$ and $\lambda(E)=2R_E$). Finally, $d_{H}(E,(x-R_E,x+R_E))\ge\rho_E-R_E$ for each~$x\in\R$, with equality if and only if $x=x_E$, hence
$$\delta_H(E)=\frac{\rho_E-R_E}{\rho_E+R_E}\ \hbox{ in dimension $N=1$}.$$

Furthermore, in any dimension $N\ge1$,~\eqref{defrhoE} implies that $\|d(x,\cdot)\|_{L^\infty(E)}\ge\rho_E$ for any $x\in\R^N$, hence $d_H(E,B_{R_E}(x))\ge\|d(\cdot,B_{R_E}(x))\|_{L^\infty(E)}\ge\rho_E-R_E$ and
\be\label{ineqdeltaH}
\delta_H(E)\ge\frac{\rho_E-R_E}{\rho_E+R_E}\ \hbox{ in any dimension $N\ge1$},
\ee
by~\eqref{defdeltaH}. Notice that, unlike the case $N=1$, the inequality~\eqref{ineqdeltaH} is in general strict in dimensions $N\ge2$. For instance, with $N\ge2$, pick any $a>0$ and consider the spherical shell
$$E:=B_{(a^N+1)^{1/N}}\setminus B_a.$$
For this set, one has $R_E=1$, $\rho_E=(a^N+1)^{1/N}$, and, since $a>(a^N+1)^{1/N}-1$ (because $N\ge 2$), it follows that
$$\min_{x\in\R^N}d_H(E,B_1(x))=d_H(E,B_1)=a>(a^N+1)^{1/N}-1=\rho_E-R_E,$$
hence $\delta_H(E)=a/(\rho_E+R_E)>(\rho_E-R_E)/(\rho_E+R_E)$.

An important consequence of~\eqref{ineqdeltaH} is that $\delta_H(E)\to1$ if $\rho_E/R_E\to+\infty$, as for the sets $E_n$ and $G_n$ defined in~\eqref{defEn} and~\eqref{defGn}, for which $\rho_{E_n}/R_{E_n}\sim n\sqrt{N}/2$ and $\rho_{G_n}/R_{G_n}\sim n|x-y|/2^{1+1/N}$ as $n\to+\infty$.

\subsubsection*{Some properties of these indices}

Let us now list some further properties satisfied by the fragmentation indices~$\delta_1$ and~$\delta_H$. First of all, they are clearly invariant by rigid motion. They are also invariant by contraction or dilation, that is,
$$\delta_1(\mu E)=\delta_1(E)\ \hbox{ and }\ \delta_H(\mu E)=\delta_H(E)$$
for every bounded non-negligible Borel set $E$ and for every $\mu>0$. 

Furthermore, the comparisons between $\delta_1$ and $\delta_H$, respectively between $1-\delta_1$ and~$1-\delta_H$, are summarized in the following proposition. In short, weakly fragmented sets for the index~$\delta_H$ are also weakly fragmented for $\delta_1$, and highly fragmented sets for the index~$\delta_1$ are also highly fragmented for $\delta_H$, whereas the reverse comparisons are false.

\begin{proposition}\label{pro1} \begin{itemize}
\item[(i)] There is a constant $\gamma>0$ such that
\be\label{compar1}
0\le\delta_1(E)\le\gamma\,\delta_H(E)
\ee
for every bounded non-negligible Borel set $E$. In particular, $\delta_1(E)\to0$ as $\delta_H(E)\to0$.
\item[(ii)] On the other hand,
$$\delta_H(E)\not\to0\ \hbox{ as $\delta_1(E)\to0$}.$$
%In particular, there is no constant $\gamma'>0$ such that $\delta_H(E)\le\gamma'\delta_1(E)$ for every Borel non-negligible bounded set $E$.
\item[(iii)] There is a constant $\eta>0$ such that
\be\label{compar2}
0<1-\delta_H(E)\le\eta\,(1-\delta_1(E))^{1/N}
\ee
for every bounded non-negligible Borel set $E$. In particular, $\delta_H(E)\to1$ as $\delta_1(E)\to1$.
\item[(iv)] On the other hand,
$$1-\delta_1(E)\not\to0\ \hbox{ as $1-\delta_H(E)\to0$}.$$
%In particular, there is no constant $\eta'>0$ such that $1-\delta_1(E)\le\eta'(1-\delta_H(E))^N$ for every Borel non-negligible bounded set $E$.
\end{itemize}
\end{proposition}

Part (iv) is immediate since the sets $G_n$ defined in~\eqref{defGn} are such that $\delta_H(G_n)\to1$ as $n\to+\infty$ and $\delta_1(G_n)\to1/2\neq1$ as $n\to+\infty$. For~(ii), consider a point $x\in\R^N$ with $|x|>1$ and the sets $H_n:=B_1\cup B_{1/n}(x)$ for $n\ge1$. There holds $\delta_1(H_n)\to0$ and~$\delta_H(H_n)\to(|x|-1)/(|x|+3)>0$ as $n\to+\infty$, yielding the desired conclusion. The proofs of parts~(i) and~(iii) are not as immediate, and are done in Subsection~\ref{sec41}.

%%%%%%%%%%%%%%%%%%%%%%%%%%%%%%%%%%%%%%%%%%%%%%%%%%%%%%%%%%
%%%%%%%%%%%%%%%%%%%%%%%%%%%%%%%%%%%%%%%%%%%%%%%%%%%%%%%%%%

\section{Influence of initial fragmentation on the large-time behavior}\label{sec3}

Let us come back to the Cauchy problem~\eqref{homo} with initial conditions $u_0=\mathds{1}_E$ as in~\eqref{defu0}, where $E\subset\R^N$ always stands for a bounded non-negligible Borel set and $f$ is a bistable-type function satisfying~\eqref{bistable}. Remember that $0<u(t,x)<1$ for all~$(t,x)\in(0,+\infty)\times\R^N$ for any such set $E$, and let us denote
$$\left\{\baa{l}
\mathcal{I}:=\big\{E:u(t,\cdot)\to1\hbox{ as $t\to+\infty$ locally uniformly in $\R^N$}\big\},\vspace{6pt}\\
\mathcal{E}:=\big\{E:u(t,\cdot)\to0\hbox{ as $t\to+\infty$ uniformly in $\R^N$}\big\},\vspace{6pt}\\
\mathcal{T}:=\big\{E:\exists\,a\in\R^N,\ u(t,\cdot)\to\Phi(\cdot+a)\hbox{ as $t\to+\infty$ uniformly in $\R^N$}\big\},\eaa\right.$$
where $\Phi:\R^N\to(0,1)$ is the unique radially symmetric stationary solution of~\eqref{homo} such that $\Phi(x)\to0$ as $|x|\to+\infty$. From the results of~\cite{p1} recalled in Section~\ref{intro}, any bounded non-negligible Borel set $E$ belongs to $\mathcal{I}\cup\mathcal{E}\cup\mathcal{T}$. The calligraphic letters $\mathcal{I}$, $\mathcal{E}$ and $\mathcal{T}$ respectively stand for invasion, extinction and threshold.

Our goal in this section is to determine sufficient conditions for $E$ to belong to~$\mathcal{I}$ or~$\mathcal{E}$, and to compare them. Notice at once from the comparison principle that if $E\subset F$ and if $F\in\mathcal{E}$ (resp. if $E\in\mathcal{I}$), then $E\in\mathcal{E}$ (resp. $F\in\mathcal{I}$). Actually, for monotone (for the inclusion) continuously increasing (for the Lebesgue measure) families of sets $E$, these sets belong to either $\mathcal{I}\cup\mathcal{E}$ up to at most one threshold value, from~\cite{p1}, and this is why we mainly focus on the conditions for which~$E\in\mathcal{I}$ or~$E\in\mathcal{E}$. We however do not consider here monotone families of sets~$E$ and we rather look for some conditions involving the fragmentation indices~$\delta_1(E)$ and~$\delta_H(E)$, for constant values of the Lebesgue measure~$\lambda(E)$. 

First of all, it turns out that the sets $\mathcal{I}$ and $\mathcal{E}$ are open for the topology generated by the balls $\mathbb{B}_r(E)$ defined in~\eqref{defB1}, as stated in the following proposition (whose proof is given in Subsection~\ref{sec42}).

\begin{proposition}\label{pro2}
For every $E\in\mathcal{I}$, there is $r>0$ such that $\mathbb{B}_r(E)\subset\mathcal{I}$, under the notation~\eqref{defB1}. Similarly, for every $E\in\mathcal{E}$, there is $s>0$ such that $\mathbb{B}_s(E)\subset\mathcal{E}$.
\end{proposition}

Secondly, it is easy to see that
\be\label{epsilonf}
\exists\,\varepsilon>0,\ \ \big(\lambda(E)\le\varepsilon\big)\Longrightarrow\big(E\in\mathcal{E}\big),
\ee
and this implication is independent of the fragmentation indices $\delta_1(E)$ and $\delta_H(E)$. Indeed, calling
\be\label{defM'f}
M':=\max_{[0,1]}|f'|
\ee
and $\varepsilon:=e^{-M'}(4\pi)^{N/2}\theta/2>0$, with $\theta\in(0,1)$ as in~\eqref{bistable}, and assuming that $\lambda(E)\le\varepsilon$, one has
$$0<u(1,x)\le\frac{e^{M'}}{(4\pi)^{N/2}}\int_Ee^{-|x-y|^2/4}dy\le\frac{e^{M'}\lambda(E)}{(4\pi)^{N/2}}\le\frac{\theta}{2}$$
for all $x\in\R^N$ from the maximum principle, hence $\|u(t,\cdot)\|_{L^\infty(\R^N)}\to0$ as $t\to+\infty$ since $f(0)=0$ and $f<0$ in $(0,\theta)$.

Thirdly, highly contracted sets of a given set lead to extinction, and highly dilated sets of a given set lead to invasion, as stated in the following proposition.

\begin{proposition}\label{pro3}
For any bounded non-negligible Borel $E\subset\R^N$, there are some real numbers $0<\underline{\mu}_{E}\le\overline{\mu}_{E}$ such that
\be\label{defunderovermu}\left\{\baa{lcl}
\big(0<\mu<\underline{\mu}_{E}\big) & \Longrightarrow & \big(\mu E\in\mathcal{E}\big),\vspace{3pt}\\
\big(\mu>\overline{\mu}_{E}\big) & \Longrightarrow & \big(\mu E\in\mathcal{I}\big).\eaa\right.
\ee
Furthermore, the sets $\big\{\mu>0:\mu E\in\mathcal{E}\big\}$ and $\big\{\mu>0:\mu E\in\mathcal{I}\big\}$ are open, and
\be\label{transition}
\big\{\mu>0:\mu E\in\mathcal{T}\big\}\neq\emptyset.
\ee
\end{proposition}

As a matter of fact,~\eqref{epsilonf} implies that $\underline{\mu}_{E}:=(\varepsilon/\lambda(E))^{1/N}>0$ satisfies the first assertion of~\eqref{defunderovermu}. If $E$ has a non-empty interior $\mathring{E}$, that is, if there are $x_0\in\R^N$ and $r>0$ such that $E\supset B_r(x_0)$, then $\mu E\supset B_{\mu r}(\mu x_0)$ and $\mu E\in\mathcal{I}$ as soon as $\mu r>R_{1}$, with $R_{1}>0$ as in~\eqref{defRalpha}, from the comparison principle and the invariance of~\eqref{homo} with respect to translations. The end of the proof of the second assertion of~\eqref{defunderovermu}, with the remaining case $\mathring{E}=\emptyset$, as well as the last part of Proposition~\ref{pro3}, are done in~Subsection~\ref{sec42}. Having in hand~\eqref{defRalpha} and Proposition~\ref{pro3}, it is tempting to conjecture that, for any bounded non-negligible Borel $E\subset\R^N$, the set $\{\mu>0:\mu E\in\mathcal{T}\}$ would be a singleton, that is, $\underline{\mu}_{E}=\overline{\mu}_{E}$ if~$\underline{\mu}_{E}$ and~$\overline{\mu}_{E}$ respectively denote the largest and the smallest real numbers satisfying~\eqref{defunderovermu}. This property is true if $E$ is further assumed to be star-shaped with respect to a point, say $x_0$, as follows from~\cite{p1}, since then the initial condition $\mathds{1}_{\mu(E-x_0)}$ are pointwise nondecreasing with respect to $\mu>0$ and increasing in $L^1(\R^N)$, and since for each $\mu>0$ the sets $\mu(E-x_0)$ and $\mu E$ belong to the same set $\mathcal{E}$, $\mathcal{T}$ or $\mathcal{I}$. The proof of the uniqueness of the element of $\{\mu>0:\mu E\in\mathcal{T}\}$ is however still open in the general case of sets $E$ that are not star-shaped.

All remaining results are concerned with the role of the fragmentation indices~$\delta_1(E)$ and~$\delta_H(E)$ on the membership of $E$ in $\mathcal{I}$ or $\mathcal{E}$. The first such result asserts that equimeasurable and highly fragmented sets for the index $\delta_1$ belong to the extinction set $\mathcal{E}$.

\begin{proposition}\label{pro4}
For each given $m>0$, there is $\varepsilon_{m}>0$ such that any bounded Borel set~$E$ satisfying $\lambda(E)=m$ and $\delta_1(E)\ge1-\epsilon_{m}$ belongs to $\mathcal{E}$.
\end{proposition}

Two comments are in order on this result, which is proved in Subsection~\ref{sec42}. First of all, the conclusion does not hold without the hypothesis that the sets have a given Lebesgue measure $m$, because of Proposition~\ref{pro3} and the invariance of $\delta_1$ with respect to dilations. More explicitly, for instance, the sets $Rn^2E_n$, with $n\ge2$, $E_n$ as in~\eqref{defEn} and $R>R_{1}$, all belong to $\mathcal{I}$ from~\eqref{defRalpha} and the comparison principle, whereas $\delta_1(Rn^2E_n)=\delta_1(E_n)\to1$ as $n\to+\infty$.

We also point out that a similar statement as Proposition~\ref{pro4} would be false if~$\delta_1$ were replaced by~$\delta_H$. For instance, consider $R>R_{1}$ and $E'_n:=B_R\cup B_1(x_n)$, with $R+1<|x_n|\to+\infty$ as $n\to+\infty$. Then, for every $n\in\N$, one has $\lambda(E'_n)=\omega_N(R^N+1)$ and $E'_n\in\mathcal{I}$ from the comparison principle, whereas $\delta_H(E'_n)\to1$ as $n\to+\infty$. On the other hand, the extinction is nevertheless possible for highly fragmented sets for the index~$\delta_H$ with fixed Lebesgue measure, even if they are not highly fragmented for the index $\delta_1$. For instance, for any $r\in(0,2R_{1})$, the sets $D_a:=(-a-r,-a)\cup(a,a+r)$, given in the first paragraph of Section~\ref{sec2}, with measure $2r$, belong to $\mathcal{E}$ for all $a>0$ large enough, while $\delta_H(D_a)\to1$ and~$\delta_1(D_a)\to1/2$ as $a\to+\infty$ (this follows from~\cite{grh}, and a similar property holds immediately in any dimension $N\ge1$).

After dealing with highly fragmented sets, let us now consider weakly fragmented sets, in the following proposition and an immediate corollary.

\begin{proposition}\label{pro5}
For each given $m>\lambda(B_{R_{1}})=\omega_NR_{1}^N$, with $R_{1}$ given in~\eqref{defRalpha}, there is $\eta_{m}>0$ such that any bounded Borel set $E$ satisfying $\lambda(E)=m$ and $\delta_1(E)\le\eta_{m}$ belongs to~$\mathcal{I}$.
\end{proposition}

Proposition~\ref{pro5} is proved in Subsection~\ref{sec42}. Together with Proposition~\ref{pro1}~(i), the following corollary immediately holds.

\begin{corollary}\label{cor1}
For each given $m>\lambda(B_{R_{1}})=\omega_NR_{1}^N$, any bounded Borel set $E$ satisfying $\lambda(E)=m$ and $\delta_H(E)\le\eta_{m}/\gamma$ belongs to~$\mathcal{I}$, where $\eta_{m}>0$ is given in Proposition~$\ref{pro5}$ and $\gamma>0$ in~\eqref{compar1}.
\end{corollary}

As for Proposition~\ref{pro4}, the conclusions of Proposition~\ref{pro5} and Corollary~\ref{cor1} immediately do not hold without the hypothesis that the sets have a given Lebesgue measure $m$. For instance, the balls $B_r$ with $0<r<R_{1}$ belong to $\mathcal{E}$ by~\eqref{defRalpha}, but nevertheless have fragmentation indices $\delta_1(B_r)$ and $\delta_H(B_r)$ equal to $0$.

Finally, after the previous results about highly or weakly fragmented sets with given Lebesgue measure, we investigate the following question: if two bounded Borel sets~$E_1$ and~$E_2$ have the same Lebesgue measure, is it possible to decide about their membership of $\mathcal{E}$ or $\mathcal{I}$ according to the comparison of the values of~$\delta_1(E_1)$ and~$\delta_1(E_2)$, or~$\delta_H(E_1)$ and~$\delta_H(E_2)$? In other words, is there a kind of monotonicity of the large-time dynamics of the solutions of~\eqref{homo} and~\eqref{defu0} with respect to the fragmentation indices $\delta_1$ or $\delta_H$ of the initial set~$E$?

Actually, the answer to this question is easily seen to be false for the index $\delta_H$. Let us explain why in this paragraph. Corollary~\ref{cor1} says that, for a given $m>\lambda(B_{R_{1}})$, the weakly fragmented bounded Borel sets $E$ for the index~$\delta_H$ (namely, $\delta_H(E)\le\eta_{m}/\gamma$) with Lebesgue measure equal to $m$ belong to the invasion set~$\mathcal{I}$. Now, call $R:=(m/\omega_N)^{1/N}>R_{1}$, pick any $e\in\R^N\setminus\{0\}$ and define the bounded Borel sets
$$O_n:=\bigcup_{k=1}^nB_{R/n^{1/N}}(ke).$$
One has $\lambda(O_n)=m$ for all $n\in\N$ large enough, and $\delta_1(O_n)\to1$ as $n\to+\infty$ (and also $\delta_H(O_n)\to1$ as $n\to+\infty$ by Proposition~\ref{pro1}). Using Proposition~\ref{pro4}, there is $n_0\ge2$ large enough such that $\lambda(O_{n_0})=m$ and $O_{n_0}\in\mathcal{E}$. Pick now any $R'\in(R_{1},R)$ and call $r'>0$ such that $\omega_N{r'}^N=m-\omega_N{R'}^N$. Choose any sequence $(x_p)_{p\in\N}$ in $\R^N$ such that $|x_p|\to+\infty$ as $p\to+\infty$. The bounded Borel sets
$$Q_p:=B_{R'}\cup B_{r'}(x_p)$$
satisfy $\lambda(Q_p)=m$ for all $p$ large enough, and they belong to $\mathcal{I}$ since they contain the ball~$B_{R'}$ with $R'>R_{1}$. Furthermore, $\lim_{p\to+\infty}\delta_H(Q_p)=1>\delta_H(O_{n_0})>0$ (the inequality $\delta_H(O_{n_0})>0$ holds since $n_0\ge2$ and thus $O_{n_0}$ is not a ball up to a negligible set). Therefore, there is $p_0\in\N$ large enough such that $\lambda(Q_{p_0})=m$ and $\delta_H(Q_{p_0})>\delta_H(O_{n_0})>0$, while $Q_{p_0}\in\mathcal{I}$. As a conclusion, the sets $B_R$, $O_{n_0}$ and $Q_{p_0}$ are equimeasurable,
$$0=\delta_H(B_R)<\delta_H(O_{n_0})<\delta_H(Q_{p_0}),$$
while
$$B_R\in\mathcal{I},\ \ O_{n_0}\in\mathcal{E}\ \hbox{ and }\ Q_{p_0}\in\mathcal{I}.$$
In other words, there is no monotonicity of the large-time dynamics of the solutions of~\eqref{homo} and~\eqref{defu0} with respect to $\delta_H(E)$ in the class of equimeasurable sets.

As far as the fragmentation index $\delta_1$ is concerned, the answer to the same monotonicity question is not that clear. For a given $m>\lambda(B_{R_{1}})$, we know from Proposition~\ref{pro5} that the weakly fragmented bounded Borel sets $E$ for the index $\delta_1$ (namely, $\delta_1(E)\le\eta_{m}$) with Lebesgue measure equal to $m$ belong to the invasion set~$\mathcal{I}$, while the highly fragmented ones (namely, $\delta_1(E)\ge1-\epsilon_{m}$) belong to the extinction set $\mathcal{E}$, from Proposition~\ref{pro4}. However, what happens for intermediate values of $\delta_1(E)$ is not as clear as with the fragmentation index $\delta_H(E)$. The last main result of the paper actually shows that there is in general no monotonicity of the large-time dynamics of the solutions of~\eqref{homo} and~\eqref{defu0} with respect to~$\delta_1(E)$ in the class of equimeasurable sets. 

\begin{theorem}\label{th1}
There are some bounded non-negligible Borel sets $E_1$ and $E_2$ such that $\lambda(E_1)=\lambda(E_2)$, $0<\delta_1(E_2)<\delta_1(E_1)$, while $E_1\in\mathcal{I}$ and $E_2\in\mathcal{E}$.
\end{theorem}

With the notations of Theorem~\ref{th1} and Proposition~\ref{pro4}, calling $m:=\lambda(E_1)=\lambda(E_2)>0$, one has $0<\delta_1(E_2)<\delta_1(E_1)<1-\epsilon_{m}$. Consider now the sets $E_n$ defined in~\eqref{defEn} (for $n\ge3$, to avoid the confusion with the above sets $E_1$ and $E_2$) and satisfying~\eqref{delta1En}. There is $n_0\ge3$ such that $\delta_1(E_{n_0})\ge1-\epsilon_{m}$. Call $F:=\mu E_{n_0}$, with $\mu=(m/\lambda(E_{n_0}))^{1/N}>0$, hence $\lambda(F)=m$, $\delta_1(F)=\delta_1(E_{n_0})\ge1-\epsilon_{m}$ and $F\in\mathcal{E}$ by Proposition~\ref{pro4}. The bounded non-negligible Borel sets $E_0$, $E_1$ and $F$ are equimeasurable and satisfy
$$\delta_1(E_2)<\delta_1(E_1)<\delta_1(F),$$
while
$$E_2\in\mathcal{E}, \ \ E_1\in\mathcal{I},\ \hbox{ and }\ F\in\mathcal{E}.$$

Theorem~\ref{th1}, sustained by numerical simulations (see Figure~\ref{fig:evolution_u}), is proved in Subsection~\ref{sec43}. We provide a proof based on some homogenization results, holding in any dimension $N\ge1$. We also give another proof, holding in dimensions $N\ge2$, based on completely different geometric arguments and the construction of suitable initial sets as intersections of balls with cubes. The first proof, based on homogenization techniques, appears to highlight an intrinsically positive effect of fragmentation on invasibility, as shown in Fig.~\ref{fig:evolution_u}. The second proof, which is more geometric, is instead based on the lack of the $\delta_1$ index to capture certain types of fragmentation, rather than on an intrinsic effect of fragmentation. Therefore, we believe that the arguments of the first proof, hence the non-monotonicity of the large-time dynamics of the solutions of~\eqref{homo} and~\eqref{defu0} with respect to the fragmentation of $E$ in the class of equimeasurable sets, should remain valid for a large class of fragmentation indices (unlike the arguments of the second proof, which should no longer hold for indices that capture the distance between the connected components of $E$ more finely, see Fig.~\ref{fig:geom} in Subsection~\ref{sec43}). Other indices could be based, for example, on the Wasserstein distance between the measures $\mathds{1}_E$ and $\mathds{1}_{B_{R_E}(x)}$.

\begin{figure}
\centering
\includegraphics[width=0.45\textwidth]{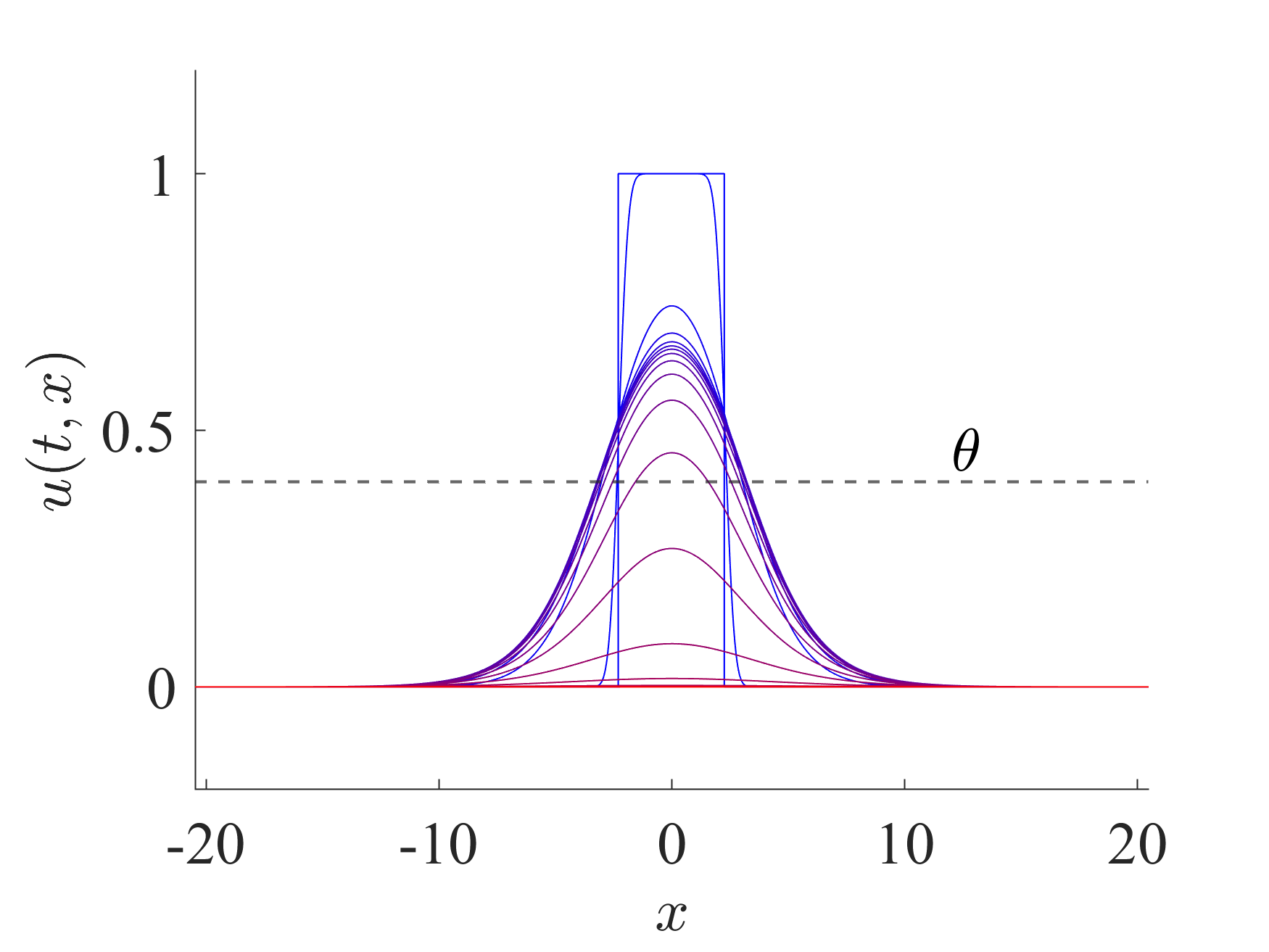}
\includegraphics[width=0.45\textwidth]{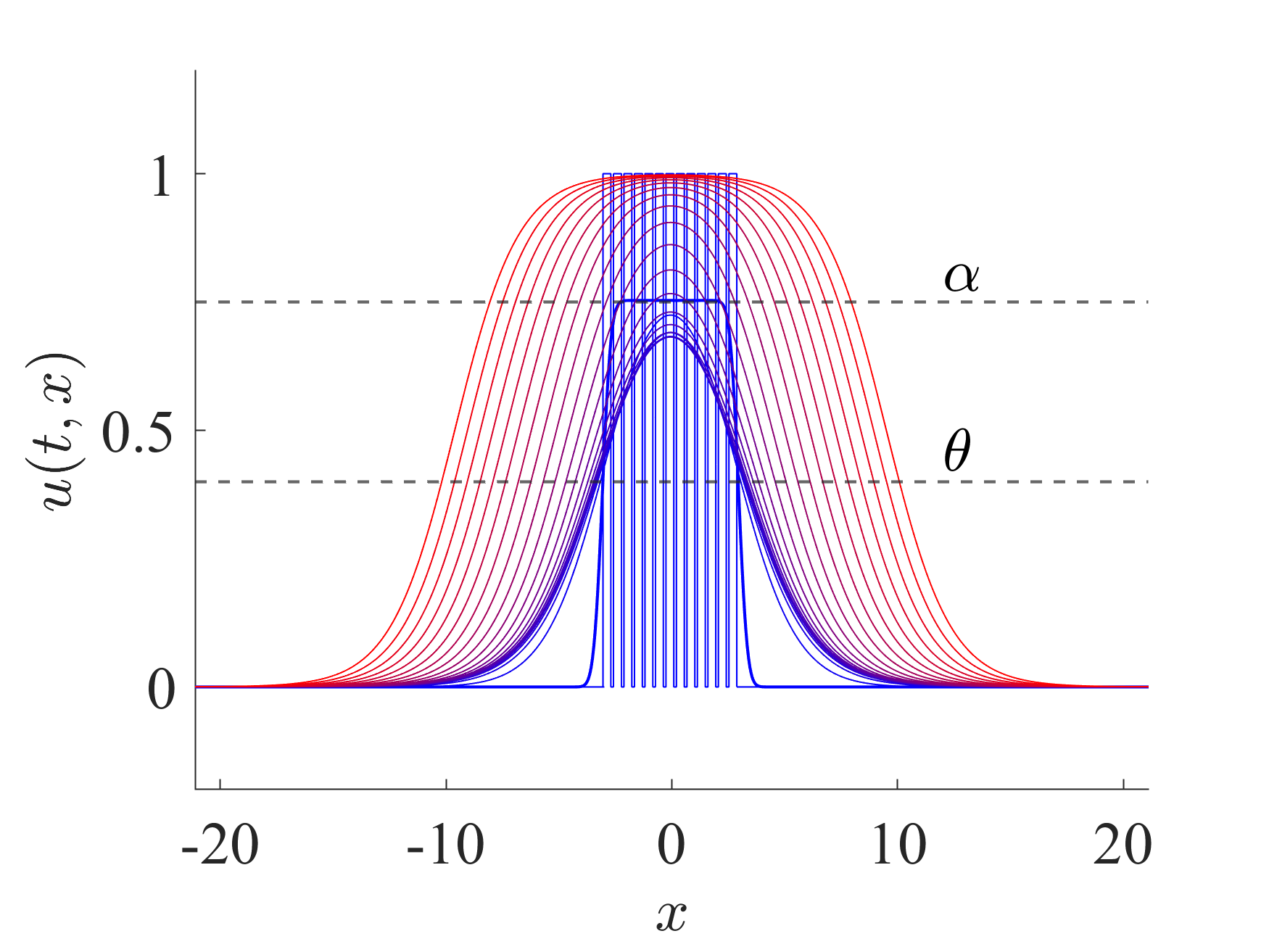}
\caption{Evolution of the numerical solution $u(t,x)$ of~\eqref{homo}, with $N=1$, $f(s)=s(1-s)(s-\theta)$, and $\theta=0.4$, starting with an initial condition $u_0(x)=\mathds{1}_E(x)$ with (left) $E=E_1:=(-L/2,L/2)$ and (right) $E=E_2:=\bigcup_{x\in\Z \cap [-k,k]}\Big[\frac{x}{z}+\Big(\!\!-\!\frac{\alpha}{2z},\frac{\alpha}{2z}\Big)\Big]$. The curves correspond to the solution  for $t=0$, $t=0.05$ and for $t$ ranging from 4 to 80 with a step size of 4. The gradient color goes from blue for the earliest times to red for the latest times. In the left panel, $\lambda(E_1)=L=4.55$ and $\delta_1(E_1)=0$; in the right panel $\alpha=3/4$, $z=2.16$, $k=6$, $\lambda(E_2)=4.52$ and $\delta_1(E_2)=0.23$. We observe here that, though $z$ is not that large and $\lambda(E_2)$ is even smaller than $\lambda(E_1)$, fragmentation (right panel) improves invasion success. We also note that $u(t,x)\approx \alpha \mathds{1}_{B_R}$ with $R=k/z+\alpha/(2 \, z)$ when $t\ll 1$. The Matlab code used for the computations is available at \url{http://doi.org/10.17605/OSF.IO/ZM479}. %*** Il faudra voir où placer la fig et où y faire référence dans le texte. Pour $E_2$ c'était assez ``sioux" car il fallait que les trous soient des multiples du pas d'espace dx et que l'on soit meilleur qu'avec une DI à support connexe de mesure supérieure, c'est pourquoi je n'ai pas pu prendre de $z$ entier. On reste néanmoins dans l'esprit des $F_n$ définis par François dans la preuve du théorème \ref{th1}. ***
}
\label{fig:evolution_u}
\end{figure}

%%%%%%%%%%%%%%%%%%%%%%%%%%%%%%%%%%%%%%%%%%%%%%%%%%%%%%%%%%
%%%%%%%%%%%%%%%%%%%%%%%%%%%%%%%%%%%%%%%%%%%%%%%%%%%%%%%%%%

\section{Proofs of the main results}\label{sec4}

%%%%%%%%%%%%%%%%%%%%%%%%%%%%%%%%%%%%%%%%%%%%%%%%%%%%%%%%%%

\subsection{Comparisons between the indices $\delta_1$ and $\delta_H$: proof of Proposition~\ref{pro1}}\label{sec41}

\begin{proof}[Proof of Proposition~$\ref{pro1}$] As already underlined in Section~\ref{sec2} after the statement of Proposition~\ref{pro1}, only parts~(i) and~(iii) remain to be proved. For the proof of~(i), one can restrict without loss of generality to the class of bounded Borel sets $E$ such that $\rho_E=1$, since the indices~$\delta_1$ and~$\delta_H$ are invariant by dilation or contraction of the sets. Consider any such set $E$. Assume first that
$$\delta_H(E)<\frac14.$$
From~\eqref{defdeltaH} and $0<R_E\le\rho_E=1$, there is then a point $y_E\in\R^N$ such that
$$d_H(E,B_{R_E}(y_E))=\delta_H(E)\,(\rho_E+R_E)\le2\delta_H(E)<\frac12,$$
hence $E\subset B_{R_E+1/2}(y_E)$ up to a negligible set. It follows that $R_E+1/2\ge\rho_E=1$, that is, $R_E\ge1/2$. On the other hand,
$$\baa{rcl}
\displaystyle\delta_1(E)\ \le\ \frac{\lambda(E\!\setminus\!B_{R_E}(y_E))}{\lambda(E)} & \le & \displaystyle\frac{\omega_N(R_E+d_H(E,B_{R_E}(y_E)))^N-\omega_NR_E^N}{\omega_NR_E^N}\vspace{3pt}\\
& \le & \displaystyle\frac{(R_E+2\delta_H(E))^N-R_E^N}{R_E^N}.\eaa$$
Since $1/2\le R_E\le R_E+2\delta_H(E)\le3/2$, one infers from the mean value theorem that $\delta_1(E)\le2^NN(3/2)^{N-1}\times(2\delta_H(E))=4N3^{N-1}\delta_H(E)$.

If now $\delta_H(E)\ge1/4$, then $\delta_1(E)<1\le4\delta_H(E)$. Finally, $\delta_1(E)\le4N3^{N-1}\delta_H(E)$ for every bounded non-negligible Borel set $E$, that is,~\eqref{compar1} holds with $\gamma:=4N3^{N-1}>0$.

For the proof of~(iii), one can restrict without loss of generality to the class of bounded Borel sets $E$ such that $R_E=\sqrt{N}/2$. Consider any such set $E$. Assume first that
$$0<1-\delta_1(E)<\frac{1}{6^N}.$$
Then
$$\frac{1}{2}\times\Big(\frac{1}{(1-\delta_1(E))^{1/N}}-1\Big)-1>\frac{1}{4}\times\frac{1}{(1-\delta_1(E))^{1/N}}>\frac{3}{2},$$
and there is an integer $m_E\ge2$ such that
\be\label{defmE}
\frac{1}{4}\times\frac{1}{(1-\delta_1(E))^{1/N}}<m_E<\frac{1}{2}\times\Big(\frac{1}{(1-\delta_1(E))^{1/N}}-1\Big).
\ee
Now, since any cube with sides of unit length is included into a ball of radius $R_E=\sqrt{N}/2$, one has in particular
$$\max_{k\in\Z^N}\frac{\lambda(E\cap(k+(0,1)^N))}{\lambda(E)}\le\max_{x\in\R^N}\frac{\lambda(E\cap B_{R_E}(x))}{\lambda(E)}=1-\delta_1(E),$$
where the last equality follows from~\eqref{defdelta1}. On the other hand, since $\lambda(E)>0$ and the cubes $k+(0,1)^N$ cover $\R^N$ up to a negligible set as $k$ describes $\Z^N$, there is $k_E\in\Z^N$ such that
$$0<\max_{k\in\Z^N}\lambda(E\cap(k+(0,1)^N))=\lambda(E\cap(k_E+(0,1)^N))\le(1-\delta_1(E))\,\lambda(E).$$
Thus,
$$\lambda\Big(E\cap\bigcup_{k\in\Z^N,\,\|k\|_\infty\le m_E}(k_E+k+(0,1)^N)\Big)\le(2m_E+1)^N\times(1-\delta_1(E))\,\lambda(E).$$
Since $\lambda(E)>0$ and $(2m_E+1)^N(1-\delta_1(E))<1$ from the right inequality in~\eqref{defmE}, it follows that
$$\lambda\Big(E\cap\Big(\R^N\setminus\bigcup_{k\in\Z^N,\,\|k\|_\infty\le m_E}(k_E+k+(0,1)^N)\Big)\Big)>0,$$
hence there is $k'_E\in\Z^N$ with $\|k'_E\|_\infty>m_E$ (that is, $\|k'_E\|_\infty\ge m_E+1$) such that $\lambda(E\cap(k_E+k'_E+(0,1)^N))>0$. Remembering that $\lambda(E\cap(k_E+(0,1)^N))>0$, one gets that $\diam(E)\ge m_E$, hence
$$\rho_E\ge\frac{\diam(E)}{2}\ge\frac{m_E}{2}\ge\frac{1}{8\,(1-\delta_1(E))^{1/N}},$$
by using~\eqref{rhodiam} and the left inequality in~\eqref{defmE}. Together with~\eqref{ineqdeltaH} and the normalization $R_E=\sqrt{N}/2$, one infers that
$$1-\delta_H(E)\le\frac{2R_E}{\rho_E+R_E}\le\frac{2R_E}{\rho_E}\le8\sqrt{N}\,(1-\delta_1(E))^{1/N}.$$

If now $1-\delta_1(E)\ge1/6^N$, then $1-\delta_H(E)\le1\le6\,(1-\delta_1(E))^{1/N}$. Finally, $1-\delta_H(E)\le8\sqrt{N}\,(1-\delta_1(E))^{1/N}$ for every bounded non-negligible Borel set $E$, that is,~\eqref{compar2} holds with $\eta:=8\sqrt{N}$. The proof of Proposition~\ref{pro1} is thereby complete.
\end{proof}

%%%%%%%%%%%%%%%%%%%%%%%%%%%%%%%%%%%%%%%%%%%%%%%%%%%%%%%%%%

\subsection{Extinction vs. invasion for dilated and highly or weakly fragmented sets}\label{sec42}

This subsection is devoted to the proofs of Propositions~\ref{pro2}-\ref{pro5} on the large-time dynamics of solutions of~\eqref{homo} for close, dilated, highly fragmented, or weakly fragmented indicator sets~$E$ in~\eqref{defu0}.

\begin{proof}[Proof of Proposition~$\ref{pro2}$.] Let us first consider $E$ in $\mathcal{I}$ and let us show that $F\in\mathcal{I}$ for any bounded Borel set $F$ such that $d_1(E,F)$ is small enough. Let $u_E$ and $u_F$ denote the solutions of~\eqref{homo} with initial conditions $\mathds{1}_E$ and $\mathds{1}_F$ respectively. Pick any $\alpha\in(\theta,1)$ with $\theta\in(0,1)$ as in~\eqref{bistable}, remember the definition of $R_{\alpha}>0$ in~\eqref{defRalpha}, and pick any~$R\in(R_{\alpha},+\infty)$. As $u_E(t,\cdot)\to1$ as $t\to+\infty$ locally uniformly in $\R^N$, there is~$T>0$ such that $\min_{\overline{B_R}}u_E(T,\cdot)>\alpha$. Since
\be\label{ineqT}
\|u_E(T,\cdot)-u_F(T,\cdot)\|_{L^\infty(\R^N)}\le\frac{e^{M'T}}{(4\pi T)^{N/2}}\times\|\mathds{1}_E-\mathds{1}_F\|_{L_1(\R^N)}
\ee
from the maximum principle, with $M'=\max_{[0,1]}|f'|$ as in~\eqref{defM'f}, there is then $r>0$ such that $\min_{\overline{B_R}}u_F(T,\cdot)\ge\alpha$ if $d_1(E,F)=\|\mathds{1}_E-\mathds{1}_F\|_{L_1(\R^N)}<r$. Therefore, for every $F\in\mathbb{B}_r(E)$, there holds $1\ge u_F(T,\cdot)\ge\alpha\mathds{1}_{B_R}$ in $\R^N$, hence $u_F(t,\cdot)\to1$ as~$t\to+\infty$ locally uniformly in $\R^N$, from the comparison principle and~\eqref{defRalpha} again. In other words,~$\mathbb{B}_r(E)\subset\mathcal{I}$.

Let us now assume that $E\in\mathcal{E}$. Therefore, there is $T>0$ such that $0\le u_E(T,\cdot)\le\theta/3$ in $\R^N$. From~\eqref{ineqT}, there is then $s>0$ such that $u_F(T,\cdot)\le\theta/2$ in $\R^N$ for every bounded Borel set $F$ satisfying $d_1(E,F)<s$, and then the nonnegative function $u_F$ converges to~$0$ as $t\to+\infty$ uniformly in $\R^N$, from~\eqref{bistable} and the comparison principle. In other words,~$\mathbb{B}_s(E)\subset\mathcal{E}$.
\end{proof}

\begin{proof}[Proof of Proposition~$\ref{pro3}$.] As already underlined in Section~\ref{sec3} after the statement of Proposition~\ref{pro3}, for the proof of~\eqref{defunderovermu}, only the case of dilated sets~$\mu E$ with large $\mu$ and $\mathring{E}=\emptyset$ remains to be dealt with. Let $E$ be such a set. Since~$\lambda(E)>0$ and since
$$\frac{1}{\lambda(B_r(x))}\int_{B_r(x)}\mathds{1}_E(y)\,dy\,\to\,\mathds{1}_E(x)\ \hbox{ as }r\mathop{\to}^>0\hbox{ for almost every $x\in\R^N$}$$
by Lebesgue's differentiation theorem, there is $x_0\in\R^N$ such that $\lambda(E\cap B_r(x_0))\sim\lambda(B_r(x_0))$ as $r\displaystyle\mathop{\to}^>0$. Since~\eqref{homo} is invariant by translation, one can assume without loss of generality that $x_0=0$. Pick any $R>R_{1}$, with $R_{1}$ defined in~\eqref{defRalpha} with $\alpha=1$ (hence, $B_R\in\mathcal{I}$). Since
$$\lambda(B_R)\ge\lambda(\mu E\cap B_R)=\mu^N\lambda(E\cap B_{R/\mu})\mathop{\sim}_{\mu\to+\infty}\mu^N\lambda(B_{R/\mu})=\lambda(B_R),$$
it follows that $d_1(B_R,\mu E\cap B_R)\to0$ as $\mu+\infty$. Therefore, $\mu E\cap B_R\in\mathcal{I}$ for all $\mu>0$ large enough, from Proposition~\ref{pro2}. Finally, $\mu E\in\mathcal{I}$ for all $\mu>0$ large enough, from the comparison principle.

Let us now turn to the proof of the openness of the sets $\{\mu>0:\mu E\in\mathcal{E}\}$ and $\{\mu>0:\mu E\in\mathcal{I}\}$. In the case when the bounded Borel set $E\subset\R^N$ has a negligible boundary for the $N$-dimensional Lebesgue measure $\lambda$, that is, if $\lambda(\partial E)=0$, then, for any~$\mu_0>0$, $\lambda(\partial(\mu_0E))=0$ and
$$\baa{rcl}
d_1(\mu E,\mu_0E) & = & \displaystyle\int_{\R^N}|\mathds{1}_{\mu E}(x)-\mathds{1}_{\mu_0E}(x)|\,dx\vspace{3pt}\\
& = & \displaystyle\int_{\widering{\mu_0E}}|\mathds{1}_{\mu E}(x)-\mathds{1}_{\mu_0E}(x)|\,dx+\int_{\R^N\setminus\overline{\mu_0 E}}|\mathds{1}_{\mu E}(x)-\mathds{1}_{\mu_0E}(x)|\,dx\ \mathop{\longrightarrow}_{\mu\to\mu_0}\ 0\eaa$$
from Lebesgue's dominated convergence theorem. Therefore, in this case, the openness of the sets $\{\mu>0:\mu E\in\mathcal{E}\}$ and $\{\mu>0:\mu E\in\mathcal{I}\}$ follows from Proposition~\ref{pro2}. 

Consider now the case of a general bounded Borel set $E\subset\R^N$. Assume first that~$\mu_0>0$ is such that
$$\mu_0E\in\mathcal{I}.$$
For $\mu>0$, let $u_\mu$ denote here the solution of~\eqref{homo} with initial condition $\mathds{1}_{\mu E}$. As in the proof of Proposition~\ref{pro2}, pick any $\alpha\in(\theta,1)$ and any $R>R'>R_{\alpha}$, with $R_{\alpha}>0$ as in~\eqref{defRalpha}. By hypothesis, there holds $u_{\mu_0}(t,\cdot)\to1$ as~$t\to+\infty$ locally uniformly in $\R^N$. There is then $T>0$ such that $\min_{\overline{B_R}}u_{\mu_0}(T,\cdot)>\alpha$. On the other hand, for every $\mu>0$, the function
$$(t,x)\mapsto v_\mu(t,x):=u_\mu\Big(\frac{\mu^2}{\mu_0^2}t,\frac{\mu}{\mu_0}x\Big)$$
ranges in $[0,1]$ and satisfies
$$\frac{\partial v_\mu}{\partial t}=\Delta v_\mu+\frac{\mu^2}{\mu_0^2}\,f(v_\mu),\ \ t>0,\ x\in\R^N$$
with initial condition $v_\mu(0,\cdot)=\mathds{1}_{\mu_0E}=u_{\mu_0}(0,\cdot)$. Therefore, the function $w_\mu:=v_\mu-u_{\mu_0}$ vanishes at time $t=0$ and satisfies
$$\frac{\partial w_\mu}{\partial t}\le\Delta w_\mu+f(v_\mu)-f(u_{\mu_0})+M\Big|\frac{\mu^2}{\mu_0^2}-1\Big|\le\Delta w_\mu+M'|w_\mu|+M\Big|\frac{\mu^2}{\mu_0^2}-1\Big|$$
in $(0,+\infty)\times\R^N$, with $M:=\max_{[0,1]}|f|$ and $M'=\max_{[0,1]}|f'|$. It then follows from the maximum principle that
$$w_\mu(t,x)\le\frac{M}{M'}\Big|\frac{\mu^2}{\mu_0^2}-1\Big|\times\big(e^{M't}-1\big)\ \hbox{ for all }t\ge 0\hbox{ and }x\in\R^N.$$
By arguing similarly with $u_{\mu_0}-v_\mu$, one gets the same bound from above for $|v_\mu-u_{\mu_0}|$. In particular, at time $t=T$, by rewriting $v_\mu$ in terms of $u_\mu$ and changing $x$ into $(\mu_0/\mu)x$, one infers that
\be\label{ineqTbis}
\Big|u_\mu\Big(\frac{\mu^2}{\mu_0^2}T,x\Big)-u_{\mu_0}\Big(T,\frac{\mu_0}{\mu}x\Big)\Big|\le\frac{M}{M'}\Big|\frac{\mu^2}{\mu_0^2}-1\Big|\times\big(e^{M'T}-1\big)\ \hbox{ for all $x\in\R^N$}.
\ee
As a consequence, remembering that $\min_{\overline{B_R}}u_{\mu_0}(T,\cdot)>\alpha$ and that $R>R'$, there is $\epsilon\in(0,\mu_0)$ such that, if $|\mu-\mu_0|\le\epsilon$, then
$$\min_{\overline{B_{R'}}}u_\mu\Big(\frac{\mu^2}{\mu_0^2}T,\cdot\Big)\ge\alpha,$$
hence $1\ge u_\mu((\mu^2/\mu_0^2)T,\cdot)\ge\alpha\mathds{1}_{B_{R'}}$ in $\R^N$. Since $R'>R_{\alpha}$, one concludes from~\eqref{defRalpha} and the maximum principle that $u_{\mu}(t,\cdot)\to1$ as $t\to+\infty$ locally uniformly in $\R^N$ for each~$\mu\in(\mu_0-\varepsilon,\mu_0+\varepsilon)$, that is, $\mu E\in\mathcal{I}$ for any such $\mu$. From the arbitrariness of $\mu_0$ such that $\mu_0E\in\mathcal{I}$, the openness of $\{\mu>0:\mu E\in\mathcal{I}\}$ has been shown. 

Let now $\mu_0>0$ be such that
$$\mu_0E\in\mathcal{E},$$
and let $T>0$ be such that $0\le u_{\mu_0}(T,\cdot)\le\theta/3$ in $\R^N$. The inequality~\eqref{ineqTbis}, which holds independently of the hypothesis $\mu_0E\in\mathcal{E}$, provides the existence of $\epsilon\in(0,\mu_0)$ such that, if $|\mu-\mu_0|\le\epsilon$, then $u_\mu((\mu^2/\mu_0^2)T,\cdot)\le\theta/2$ in $\R^N$, hence $u_\mu(t,\cdot)\to0$ as $t\to+\infty$ uniformly in $\R^N$, from~\eqref{bistable} and the comparison principle. Finally, $\mu E\in\mathcal{E}$ for all $\mu\in(\mu_0-\epsilon,\mu_0+\epsilon)$, and the set $\{\mu>0:\mu E\in\mathcal{E}\}$ is open.

Lastly, property~\eqref{transition} immediately follows from~\eqref{defunderovermu} and the openness of the sets~$\{\mu>0:\mu E\in\mathcal{E}\}$ and $\{\mu>0:\mu E\in\mathcal{I}\}$, together with the fact that $\mu E\in\mathcal{E}\cup\mathcal{T}\cup\mathcal{I}$ for every $\mu>0$. The proof of Proposition~\ref{pro3} is thereby complete.
\end{proof}

\begin{proof}[Proof of Proposition~$\ref{pro4}$.] Let $m>0$ be given, define
$$s:=\frac{1}{\sqrt{N}}\times\Big(\frac{m}{\omega_N}\Big)^{1/N}>0,$$
and observe that any cube of measure $(2s)^N$ is included into a ball of measure $m$. Consider now any bounded Borel set $E$ with $\lambda(E)=m$, and let $u$ be the solution of~\eqref{homo} with initial condition~\eqref{defu0}. Formula~\eqref{defdelta1} and the previous observations then imply that~$\lambda(Q\cap E)\le m\,(1-\delta_1(E))$ for every cube~$Q$ of measure $(2s)^N$. Therefore, remembering the definition $M'=\max_{[0,1]}|f'|$, it follows from the comparison principle that, for every~$x\in\R^N$,
$$\baa{rcl}
\displaystyle 0\ \le\ u(1,x) & \le & \displaystyle \frac{e^{M'}}{(4\pi)^{N/2}}\times\int_Ee^{-|x-y|^2/4}dy\vspace{3pt}\\
& = & \displaystyle\frac{e^{M'}}{(4\pi)^{N/2}}\times\sum_{k\in\Z^N}\int_{E\cap(x+2ks+(-s,s)^N)}e^{-|x-y|^2/4}dy\vspace{3pt}\\
& \le & \displaystyle\frac{e^{M'}}{(4\pi)^{N/2}}\times m\times(1-\delta_1(E))\times\sum_{k\in\Z^N}e^{-\max(2\|k\|_\infty-1;0)^2s^2/4},\eaa$$
where the above two series converge. Since the right-hand side of the last inequality does not depend on $x$, there is then $\epsilon_{m}>0$ such that $0\le u(1,\cdot)\le\theta/2$ in $\R^N$ as soon as~$\delta_1(E)\ge1-\epsilon_{m}$, hence $u(t,\cdot)\to0$ as $t\to+\infty$ uniformly in $\R^N$. In other words, every bounded Borel set $E$ such that $\lambda(E)=m$ and $\delta_1(E)\ge1-\epsilon_{m}$ belongs to $\mathcal{E}$.
\end{proof}

\begin{proof}[Proof of Proposition~$\ref{pro5}$.] Fix any $m>\lambda(B_{R_{1}})=\omega_NR_{1}^N$. Define $\alpha=(\theta+1)/2\in(\theta,1)$ and $R=R_{\alpha}+1$, with $\theta\in(0,1)$ and $R_{\alpha}$ given in~\eqref{bistable} and~\eqref{defRalpha}. Let $v$ denote the solution of~\eqref{homo} with initial condition $\mathds{1}_{B_{(m/\omega_N)^{1/N}}}$. Since $(m/\omega_N)^{1/N}>R_{1}$ by assumption, it follows from~\eqref{defRalpha} that $v(t,\cdot)\to1$ as $t\to+\infty$ locally uniformly in $\R^N$, hence there is~$T>0$ such that $\min_{\overline{B_R}}v(T,\cdot)>\alpha$. Consider now any bounded Borel set $E$ such that~$\lambda(E)=m$ and, from~\eqref{defdelta1}, let $z_E\in\R^N$ be such that
$$\|\mathds{1}_E-\mathds{1}_{B_{(m/\omega_N)^{1/N}}(z_E)}\|_{L_1(\R^N)}=2\lambda(E\setminus B_{(m/\omega_N)^{1/N}}(z_E))=2\lambda(E)\delta_1(E)=2m\delta_1(E).$$
As in~\eqref{ineqT}, the solution $u_E$ of~\eqref{homo} with initial condition $\mathds{1}_E$ satisfies
$$\|u_E(T,\cdot)-v(T,\cdot-z_E)\|_{L^\infty(\R^N)}\le\frac{e^{M'T}}{(4\pi T)^{N/2}}\times\|\mathds{1}_E-\mathds{1}_{B_{(m/\omega_N)^{1/N}}(z_E)}\|_{L_1(\R^N)},$$
hence $\|u_E(T,\cdot+z_E)-v(T,\cdot)\|_{L^\infty(\R^N)}\le 2e^{M'T}(4\pi T)^{-N/2}m\times\delta_1(E)$. Observing that~$\alpha$ and~$T$ only depend on~$f$ and $m$ (and on the dimension~$N$), and remembering that~$\min_{\overline{B_R}}v(T,\cdot)>\alpha$, there is then $\eta_{m}>0$ such that $\min_{\overline{B_R}}u_E(T,\cdot+z_E)\ge\alpha$ as soon as $\delta_1(E)\le\eta_{m}$. For any such $E$, one then has $u_E(T,\cdot+z_E)\ge\alpha\mathds{1}_{B_R}$, hence $u_E(t,\cdot)\to1$ as $t\to+\infty$ locally uniformly in $\R^N$, by~\eqref{defRalpha} and the definition $R=R_{\alpha}+1$. Therefore,  any bounded Borel set $E$ such that $\lambda(E)=m$ and $\delta_1(E)\le\eta_{m}$ belongs to $\mathcal{I}$.
\end{proof}

%%%%%%%%%%%%%%%%%%%%%%%%%%%%%%%%%%%%%%%%%%%%%%%%%%%%%%%%%%

\subsection{Proof of Theorem~\ref{th1}}\label{sec43}

We provide two different proofs of Theorem~\ref{th1}. The first one holds in any dimension $N\ge1$ and is based on homogenization techniques, and the constructed sets $E_1$ and $E_2$ have similarities with the sets~$F_n$ defined in~\eqref{defFn}. The alternate proof holds in dimensions $N\ge2$, and the constructed sets involve cubes and their intersections with balls.

\begin{proof}[Proof of Theorem~$\ref{th1}$ in any dimension $N\ge1$]
Consider any $\alpha\in(\theta,1)$, with $\theta\in(0,1)$ as in~\eqref{bistable}. Fix then some real numbers $R$ and $R'$, close enough to $R_\alpha$, such that
\be\label{defRR'0}
0<R'<R_\alpha<R,\ \ \alpha^{1/N}R<R',\ \hbox{ and }\ 0<\alpha\,R^N-\alpha\,{R'}^N<\min\Big(\frac{R_1^N}{2^N},\alpha^2R^N\big),
\ee
with $R_\alpha>0$ and $R_1>0$ as in~\eqref{defRalpha}. Owing to the definition of $R_\alpha$, the solution $\tilde{v}$ of~\eqref{homo} with initial condition $\alpha\,\mathds{1}_{B_{R'}}$ converges to $0$ as $t\to+\infty$ uniformly in $\R^N$, thus there is $T>0$ such that $\tilde{v}(T,\cdot)\le\theta/3$ in $\R^N$. For $\beta\in(\alpha,1]$, let $\tilde{v}^\beta$ be the solution of~\eqref{homo} with initial condition $\beta\,\mathds{1}_{B_{R'}}$. Since $0\le \tilde{v}^\beta(T,\cdot)-\tilde{v}(T,\cdot)\le e^{M'T}(4\pi T)^{-N/2}\|\beta\,\mathds{1}_{B_{R'}}-\alpha\,\mathds{1}_{B_{R'}}\|_{L^1(\R^N)}$ from the maximum principle, with $M'=\max_{[0,1]}|f'|$ as in~\eqref{defM'f}, one can choose $\beta\in(\alpha,1)$, close enough to $\alpha$, so that $0\le\tilde{v}^\beta(T,\cdot)\le\theta/2$ in~$\R^N$, hence
\be\label{vbeta}
\tilde{v}^\beta(t,\cdot)\to0\hbox{ as $t\to+\infty$ uniformly in~$\R^N$}
\ee
by~\eqref{bistable} and the maximum principle. Even if it means decreasing $\beta$ in $(\alpha,1)$, one can also assume without loss of generality that
\be\label{defRR'}
0<\alpha\,R^N-\beta\,{R'}^N<\min\Big(\frac{R_1^N}{2^N},\alpha^2R^N\Big).
\ee

Consider the following bounded Borel sets
$$F_n:=\bigcup_{z\in(\Z^N\!/n)\cap B_{R}}\Big[z+\Big(\!\!-\!\frac{\alpha^{1/N}}{2n},\frac{\alpha^{1/N}}{2n}\Big)^N\Big]$$
and
$$G_n:=\bigcup_{z\in(\Z^N\!/n)\cap B_{R'}}\Big[z+\Big(\!\!-\!\frac{\beta^{1/N}}{2n},\frac{\beta^{1/N}}{2n}\Big)^N\Big]$$
for $n\ge 1$. The sets $F_n$ and $G_n$ are such that (see~\eqref{defFn} and below)
\be\label{limits1}
\lim_{n\to+\infty}\lambda(F_n)=\alpha\,\omega_NR^N>\beta\,\omega_N{R'}^N=\lim_{n\to+\infty}\lambda(G_n)
\ee
and
\be\label{FnGn}
\lim_{n\to+\infty}\delta_1(F_n)=1-\alpha>1-\beta=\lim_{n\to+\infty}\delta_1(G_n)>0.
\ee
Together with~\eqref{defRR'}, there is then $n_1\in\N$ such that
\be\label{defn1}
0<\lambda(G_n)<\lambda(F_n)\ \hbox{ and }\ \lambda(F_n)-\lambda(G_n)<\min\Big(\omega_N\frac{R_1^N}{2^N},\alpha^2\omega_NR^N\Big)
\ee
for all $n\ge n_1$.

Let then $u_n$, $v_n$, $u$ and $v=\tilde{v}^\beta$ be the solutions of~\eqref{homo} with respective initial conditions $\mathds{1}_{F_n}$, $\mathds{1}_{G_n}$, $\alpha\,\mathds{1}_{B_R}$ and $\beta\,\mathds{1}_{B_{R'}}$. Let similarly $U_n$, $V_n$, $U$ and $V$ be the solutions of the heat equation $\frac{\partial z}{\partial t}=\Delta z$ with respective initial conditions $\mathds{1}_{F_n}$, $\mathds{1}_{G_n}$, $\alpha\,\mathds{1}_{B_R}$ and $\beta\,\mathds{1}_{B_{R'}}$. We point out that these 8 functions are nonnegative, and that $u_n$, $v_n$, $u$ and $v$ range in $[0,1]$. Since~$\|\alpha\mathds{1}_{B_R}-(\alpha-\epsilon)\mathds{1}_{B_{R-\epsilon}}\|_{L^1(\R^N)}\to0$ as $\epsilon\to0$, and since $u(t,\cdot)\to1$ as $t\to+\infty$ locally uniformly in $\R^N$ by~\eqref{defRalpha} (since $R>R_\alpha)$, it follows as in the proof of Proposition~\ref{pro2} that there is $\epsilon_1\in(0,R)$ such that the solution $u^{\epsilon_1}$ of~\eqref{homo} with initial condition
$$u^{\epsilon_1}_0:=(\alpha-\epsilon_1)\mathds{1}_{B_{R-\epsilon_1}}$$
converges to $1$ as $t\to+\infty$ locally uniformly in $\R^N$. Notice also that $U(t,\cdot)\to\alpha$ as $t\to0$ locally uniformly in $B_R$. As a consequence, there is $t_1>0$ small enough such that
$$U(t_1,\cdot)\ge\Big(\alpha-\frac{\epsilon_1}{3}\Big)\mathds{1}_{B_{R-\epsilon_1}}\ \hbox{ in }\R^N.$$
Even if it means decreasing $t_1>0$, one can assume without loss of generality that~$(\alpha-\epsilon_1/2)e^{-M't_1}\ge\alpha-\epsilon_1$. Now, owing to the definition of $F_n$, it follows by homo\-genization that
$$\baa{rcl}
0\ \le\ U_n(t_1,x) & = & \displaystyle\frac{1}{(4\pi t_1)^{N/2}}\sum_{z\in(\Z^N\!/n)\cap B_R}\int_{z+(-\alpha^{1/N}\!/(2n),\,\alpha^{1/N}\!/(2n))^N}e^{-|x-y|^2/(4t_1)}\,dy\vspace{3pt}\\
& \displaystyle\mathop{\longrightarrow}_{n\to+\infty} & \displaystyle\frac{1}{(4\pi t_1)^{N/2}}\int_{B_R}\alpha\,e^{-|x-y|^2/(4t_1)}\,dy=U(t_1,x)\eaa$$
uniformly with respect to $x\in\R^N$. Thus, there is $n_2\ge n_1$ such that
$$U_n(t_1,\cdot)\ge\Big(\alpha-\frac{\epsilon_1}{2}\Big)\mathds{1}_{B_{R-\epsilon_1}}\ \hbox{ in }\R^N\hbox{ for all }n\ge n_2.$$
Since $1\ge u_n(t_1,\cdot)\ge e^{-M't_1}U_n(t_1,\cdot)$ in $\R^N$ from the maximum principle, one infers that~$1\ge u_n(t_1,\cdot)\ge(\alpha-\epsilon_1)\mathds{1}_{B_{R-\epsilon_1}}=u^{\epsilon_1}_0$ in $\R^N$ for all $n\ge n_2$, hence $u_n(t,\cdot)\to1$ as~$t\to+\infty$ locally uniformly in $\R^N$. In other words,
\be\label{FnI}
F_n\in\mathcal{I}\ \hbox{ for all $n\ge n_2$}.
\ee

Consider now the solution $v=\tilde{v}^\beta$ of~\eqref{homo} with initial condition $\beta\,\mathds{1}_{B_{R'}}$. By~\eqref{vbeta}, there is $T'>0$ such that $0\le v(T',\cdot)\le\theta/3$ in $\R^N$. For $\epsilon\in(0,1-\beta)$, let $v^\epsilon$ be the solution of~\eqref{homo} with initial condition
$$v^\epsilon_0:=(\beta+\epsilon)\,\mathds{1}_{B_{R'+\epsilon}}+\epsilon\,\mathds{1}_{\R^N\setminus B_{R'+\epsilon}}.$$
Since
$$\baa{r}
\displaystyle 0\!\le\!v^\epsilon(T',x)\!-\!v(T',x)\!\le\!\frac{e^{M'T'}}{(4\pi T')^{N/2}}\Big[\int_{B_{R'}}\!\!\!\epsilon\,e^{-|x-y|^2/(4T')}dy+\!\!\int_{B_{R'+\epsilon}\setminus B_{R'}}\!\!\!(\beta\!+\!\epsilon)\,e^{-|x-y|^2/(4T')}dy\vspace{3pt}\\
\displaystyle+\int_{\R^N\setminus B_{R'+\epsilon}}\epsilon\,e^{-|x-y|^2/(4T')}dy\Big]\eaa$$
for all $x\in\R^N$, one infers that $v^\epsilon(T',\cdot)\to v(T',\cdot)$ uniformly in $\R^N$ as $\epsilon\to0$. Thus, there is $\epsilon_2\in(0,1-\beta)$ such that~$0\le v^{\epsilon_2}(T',\cdot)\le\theta/2$ in $\R^N$, hence $v^{\epsilon_2}(t,\cdot)\to0$ as $t\to+\infty$ uniformly in~$\R^N$. On the other hand, $\sup_{B_{R'+\epsilon_2}}v(t,\cdot)\to\beta$ and $\sup_{\R^N\setminus B_{R'+\epsilon_2}}v(t,\cdot)\to0$ as $t\to0$, while~$V(t,\cdot)\le e^{M't}v(t,\cdot)$ in $\R^N$ for all $t\ge0$, from the maximum principle. As a consequence, there is~$t_2>0$ (independent of $n$) small enough such that
$$V(t_2,\cdot)\le\Big(\beta+\frac{\epsilon_2}{3}\Big)\mathds{1}_{B_{R'+\epsilon_2}}+\frac{\epsilon_2}{3}\mathds{1}_{\R^N\setminus B_{R'+\epsilon_2}}.$$
Even if it means decreasing $t_2>0$, one can also assume without loss of generality that~$(\beta+\epsilon_2/2)e^{M't_2}\le\beta+\epsilon_2$ and $(\epsilon_2/2)e^{M't_2}\le\epsilon_2$. Now, owing to the definition of~$G_n$, there holds that $V_n(t_2,\cdot)\to V(t_2,\cdot)$ uniformly in $\R^N$ as $n\to+\infty$ by homogenization, as for $U_n(t_1,\cdot)$ and $U(t_1,\cdot)$ above. Therefore, there is~$n_3\ge n_2$ such that
$$V_n(t_2,\cdot)\le\Big(\beta+\frac{\epsilon_2}{2}\Big)\mathds{1}_{B_{R'+\epsilon_2}}+\frac{\epsilon_2}{2}\,\mathds{1}_{\R^N\setminus B_{R'+\epsilon_2}}\ \hbox{ in }\R^N\hbox{ for all }n\ge n_3.$$
As $0\le v_n(t_2,\cdot)\le e^{M't_2}V_n(t_2,\cdot)$ in $\R^N$ from the maximum principle, one infers that~$0\le v_n(t_2,\cdot)\le(\beta+\epsilon_2)\mathds{1}_{B_{R'+\epsilon_2}}+\epsilon_2\mathds{1}_{\R^N\setminus B_{R'+\epsilon_2}}=v^{\epsilon_2}_0$ in $\R^N$ for all $n\ge n_3$, hence
\be\label{vn}
0\le v_n(t_2+t,\cdot)\le v^{\epsilon_2}(t,\cdot)\hbox{ in $\R^N$ for all $t\ge0$ and $n\ge n_3$}
\ee
by the maximum principle. In particular, $v_n(t,\cdot)\to0$ as $t\to+\infty$ uniformly in $\R^N$, that is, $G_n\in\mathcal{E}$, for all $n\ge n_3$.

Remember~\eqref{defn1} and define
\be\label{defrhon}
0<\rho_n:=\Big(\frac{\lambda(F_n)-\lambda(G_n)}{\omega_N}\Big)^{1/N}<\frac{R_1}{2}
\ee
for $n\ge n_1$. Consider finally a sequence of points $(x_n)_{n\ge n_1}$ in $\R^N$ such that $B_{\rho_n}(x_n)\cap G_n=\emptyset$ for every $n\ge n_1$, and $|x_n|\to+\infty$ as $n\to+\infty$. The bounded Borel sets
$$H_n:=G_n\cup B_{\rho_n}(x_n)$$
satisfy
\be\label{HnFn}
\lambda(H_n)=\lambda(F_n)>0\ \hbox{ for all $n\ge n_1$}.
\ee
Since $\rho_n<R_1/2$, all the balls $B_{\rho_n}(x_n)$ belong to $\mathcal{E}$, by definition of $R_1$.

For the Cauchy problem~\eqref{homo} with the initial condition~$\mathds{1}_{H_n}$, the two components $\mathds{1}_{G_n}$ and $\mathds{1}_{B_{\rho_n}(x_n)}$ act as almost independent initial conditions for $n$ large enough and then~$H_n$ will belong to $\mathcal{E}$ for every $n$ large enough (as in the example of the sets $D_a$ given in the first paragraph of Section~\ref{sec2}). More precisely, to show this fact, fix first $\rho>0$ such that $G_n\subset B_\rho$ and $\rho_n\le\rho$ for all $n\ge n_1$. Denote $w$, $w_n$, and $z_n$, the solutions of~\eqref{homo} with initial conditions $\mathds{1}_{B_{R_1/2}}$, $\mathds{1}_{B_{\rho_n}(x_n)}$, and $\mathds{1}_{H_n}$, respectively. From~\eqref{vn}-\eqref{defrhon} and $\lim_{t\to+\infty}\|v^{\epsilon_2}(t,\cdot)\|_{L^\infty(\R^N)}=0$, together with the definition of $R_1$, there is $\tau>0$ such that
\be\label{deftau}
0\le v_n(\tau,\cdot)\le\frac{\theta}{3}\ \hbox{ and }\ 0\le w_n(\tau,\cdot)\le w(\tau,\cdot-x_n)\le\frac{\theta}{3}\ \hbox{ in $\R^N,\ $ for all $n\ge n_3$}.
\ee
Furthermore, since $H_n=G_n\cup B_{\rho_n}(x_n)$, the maximum principle yields
$$\baa{rcl}
0\le v_n(\tau,x)\le z_n(\tau,x) & \!\!\le\!\! & \displaystyle v_n(\tau,x)+\frac{e^{M'\tau}}{(4\pi\tau)^{N/2}}\int_{B_{\rho_n}(x_n)}e^{-|x-y|^2/(4\tau)}dy\vspace{3pt}\\
& \!\!\le\!\! & \displaystyle v_n(\tau,x)+\frac{e^{M'\tau}}{(4\pi\tau)^{N/2}}\int_{B_{\rho}}e^{-|x-x_n-y|^2/(4\tau)}dy\eaa$$
and
$$\baa{rcl}
0\le w_n(\tau,x)\le z_n(\tau,x) & \!\!\le\!\! & \displaystyle w_n(\tau,x)+\frac{e^{M'\tau}}{(4\pi\tau)^{N/2}}\int_{G_n}e^{-|x-y|^2/(4\tau)}dy\vspace{3pt}\\
& \!\!\le\!\! & \displaystyle v_n(\tau,x)+\frac{e^{M'\tau}}{(4\pi\tau)^{N/2}}\int_{B_{\rho}}e^{-|x-y|^2/(4\tau)}dy\eaa$$
for all $x\in\R^N$ and $n\ge n_1$. Let $\sigma>0$ be such that
$$\frac{e^{M'\tau}}{(4\pi\tau)^{N/2}}\int_{B_{\rho}}e^{-|\xi-y|^2/(4\tau)}dy\le\frac{\theta}{6}\ \hbox{ for all }|\xi|\ge\sigma.$$
Together with~\eqref{deftau}, it follows that $0\le z_n(\tau,x)\le\theta/2$ for all $n\ge n_3$ and for all $x\in\R^n$ such that either $|x-x_n|\ge\sigma$ or $|x|\ge\sigma$. Since $|x_n|\to+\infty$ as $n\to+\infty$, there is $n_4\ge n_3$ such that $0\le z_n(\tau,x)\le\theta/2$ for all $n\ge n_4$ and $x\in\R^N$, and then $z_n(t,\cdot)\to0$ as $t\to+\infty$ uniformly in $\R^N$. In other words,
\be\label{HnE}
H_n\in\mathcal{E}\ \hbox{ for all $n\ge n_4$}.
\ee

Let us finally estimate $\delta_1(H_n)$ for large $n$. Since $0<\lambda(H_n)=\lambda(F_n)\to\alpha\omega_NR^N$ as $n\to+\infty$ by~\eqref{limits1}, one has
\be\label{RHn}
0<R_{H_n}\to\alpha^{1/N}R<R'\ \hbox{ as $n\to+\infty$},
\ee
where the inequality $\alpha^{1/N}R<R'$ holds by~\eqref{defRR'0}. From this, $G_n\subset B_\rho$, $|x_n|\to+\infty$ and $\rho_n<R_1/2$, there is then $n_5\ge n_4$ such that, for every $n\ge n_5$ and $y\in\R^N$, one has either $G_n\cap B_{R_{H_n}}(y)=\emptyset$ or $B_{\rho_n}(x_n)\cap B_{R_{H_n}}(y)=\emptyset$. Consequently,
$$\max_{y\in\R^N}\lambda\big(H_n\cap B_{R_{H_n}}(y)\big)=\max\Big(\max_{y\in\R^N}\lambda\big(G_n\cap B_{R_{H_n}}(y)\big),\max_{y\in\R^N}\lambda\big(B_{\rho_n}(x_n)\cap B_{R_{H_n}}(y)\big)\Big)$$
for all $n\ge n_5$. But $\lambda(B_{\rho_n}(x_n))=\lambda(F_n)-\lambda(G_n)<\alpha^2\omega_NR^N<\alpha\beta\omega_NR^N$ for all $n\ge n_1$ by~\eqref{defn1}, while $\lambda(G_n\cap B_{R_{H_n}})\to\beta\omega_N\alpha R^N$ as $n\to+\infty$ by~\eqref{RHn} and the definition of $G_n$. Therefore, there is $n_6\ge n_5$ such that
$$\max_{y\in\R^N}\lambda\big(H_n\cap B_{R_{H_n}}(y)\big)=\max_{y\in\R^N}\lambda\big(G_n\cap B_{R_{H_n}}(y)\big)$$
for all $n\ge n_6$, and then
$$\lim_{n\to+\infty}\max_{y\in\R^N}\lambda\big(H_n\cap B_{R_{H_n}}(y)\big)=\lim_{n\to+\infty}\lambda(G_n\cap B_{R_{H_n}})=\alpha\beta\omega_NR^N.$$
Therefore,
$$\delta_1(H_n)\to1-\frac{\alpha\beta\omega_NR^N}{\alpha\omega_NR^N}=1-\beta\ \hbox{ as }n\to+\infty.$$
From~\eqref{FnGn}, there is then $n_7\ge n_6$ such that
$$\delta_1(F_n)>\delta_1(H_n)>0\ \hbox{ for all $n\ge n_7$}.$$
Together with~\eqref{FnI},~\eqref{HnFn} and~\eqref{HnE}, one concludes that, for each $n\ge n_7$, the sets $E_1:=F_n$ and $E_2:=H_n$ satisfy the desired properties of Theorem~\ref{th1}, completing its proof.
\end{proof}

\begin{proof}[Alternate proof of Theorem~$\ref{th1}$ in dimensions $N\ge2$]
We assume here that $N\ge2$. For $a>0$, let
$$Q_a:=\Big(-\frac{a}{2},\frac{a}{2}\Big)^N$$
be the cube of $\R^N$ centered at the origin and with sides of length $a$. For the solutions of~\eqref{homo} with the initial condition~$\mathds{1}_{Q_a}$, it follows from~\cite{p1} that there is a unique threshold~$a^*>0$ between extinction and invasion, that is,
\be\label{defa*}
Q_a\in\mathcal{E}\hbox{ if $0<a<a^*$},\ \ Q_{a^*}\in\mathcal{T},\ \hbox{ and }\ Q_a\in\mathcal{I}\hbox{ if $a>a^*$}.
\ee

For $a>0$ and $r>0$, define
$$C_{a,r}=Q_a \cap B_r.$$
For each fixed $a>0$, the family $C_{a,r}$ is continuously increasing in $L^1(\R^N)$ with respect to $r\in(0,a\sqrt{N}/2]$, while, for each fixed $r>0$, the family $C_{a,r}$ is continuously increasing in~$L^1(\R^N)$ with respect to $a\in(0,2r]$. Consider
$$a=a^*+\epsilon,$$
with $\epsilon>0$. For $r\ge a\sqrt{N}/2=(a^*+\epsilon)\sqrt{N}/2$, one has $C_{a,r}=Q_a$, hence by~\eqref{defa*} the solution of~\eqref{homo} with initial condition $\mathds{1}_{C_{a,r}}$ converges to $1$ as $t\to+\infty$ locally uniformly in~$\R^N$. On the other hand, for $r>0$ small enough, one has $C_{a,r}=B_r$ and extinction occurs, that is, $C_{a,r}\in\mathcal{E}$. Thus, by~\cite{p1} again, there is a threshold value
$$r^*(\varepsilon)\in\Big(0,\frac{a\sqrt{N}}{2}\Big)=\Big(0,\frac{(a^*+\epsilon)\sqrt{N}}{2}\Big)$$
such that
$$C_{a,r}\in\mathcal{E}\hbox{ if $0<r<r^*(\epsilon)$},\ \ C_{a,r^*(\epsilon)}\in\mathcal{T},\ \hbox{ and }\ C_{a,r}\in\mathcal{I}\hbox{ if $r>r^*(\epsilon)$}.$$ 

When $\varepsilon\to 0$, we observe that $r^*(\varepsilon)\to a^*\sqrt{N}/2$. Indeed, on the one hand, $r^*(\epsilon)<(a^*+\epsilon)\sqrt{N}/2$ for each $\epsilon>0$, hence $\limsup_{\epsilon\to0,\,\epsilon>0}r^*(\epsilon)\le a^*\sqrt{N}/2$. On the other hand, for each $r\in(0,a^*\sqrt{N}/2)$, one has $C_{a^*,r}\subset Q_{a^*}$ and $\lambda(Q_{a^*}\setminus C_{a^*,r})>0$, hence~$C_{a^*,r}\in\mathcal{E}$ by~\cite{p1} and then~$C_{a^*+\epsilon,r}\in\mathcal{E}$ for all $\epsilon>0$ small enough by Proposition~\ref{pro2}. Finally, for each $r\in(0,a^*\sqrt{N}/2)$, one has~$r<r^*(\epsilon)$ for all $\epsilon>0$ small enough, hence $\liminf_{\epsilon\to0,\,\epsilon>0}r^*(\epsilon)\ge a^*\sqrt{N}/2$, and therefore
\be\label{r*0}
\lim_{\epsilon\to0,\ \epsilon>0}r^*(\epsilon)= \frac{a^*\sqrt{N}}{2}.
\ee
When $\epsilon>0$ is large enough, so that $Q_{a^*+\epsilon}\supset B_{R_{1}}$ with $R_{1}>0$ as in~\eqref{defRalpha}, then~$r^*(\epsilon)=R_{1}$. As the family $C_{a^*+\epsilon,r}$ is nondecreasing with respect to $\varepsilon>0$, the function $\varepsilon\mapsto r^*(\varepsilon)$ is nonincreasing in $(0,+\infty)$. This map is also continuous in~$(0,+\infty)$, with similar arguments as above and by using Proposition~\ref{pro2} again. Notice also that~$\omega_N=\lambda(B_1)>(2/\sqrt{N})^N$ (because $B_1\supset Q_{2/\sqrt{N}}$ and $\lambda(B_1\setminus Q_{2/\sqrt{N}})>0$, the assumption~$N\ge2$ is used here!). One can then choose $\sigma\in(1/\omega_N^{1/N},\sqrt{N}/2)$ sufficiently close to~$\sqrt{N}/2$ so that
$$0<\lambda(Q_{a^*})-\lambda(C_{a^*,\sigma a^*})<\frac{\omega_N{a^*}^N}{2^N}.$$
As $0<\sigma<\sqrt{N}/2$, it follows from~\eqref{r*0} and the above continuity and monotonicity properties of the map $\epsilon\mapsto r^*(\epsilon)$, that there is a unique $\epsilon^*>0$ such that
$$r^*(\varepsilon^*)=\sigma\,(a^*+\varepsilon^*).$$

Fix now $\beta\in(0,1)$ small enough so that
\be\label{defbeta}
0<\lambda(Q_{a^*+\beta\epsilon^*})-\lambda(C_{a^*+\beta\epsilon^*,\sigma(a^*+\beta\epsilon^*)})<\frac{\omega_N{a^*}^N}{2^N},
\ee
define
\be\label{defr}
0<\frac{a^*+\beta\epsilon^*}{\omega_N^{1/N}}<r:=\sigma\,(a^*+\beta\epsilon^*)<\frac{(a^*+\beta\epsilon^*)\sqrt{N}}{2}
\ee
and choose $\eta\in(0,\beta)$ close enough to $\beta$ so that
\be\label{defeta1}
0<\lambda(Q_{a^*+\eta\epsilon^*})-\lambda(C_{a^*+\beta\epsilon^*,r})<\frac{\omega_N{a^*}^N}{2^N},
\ee
and
\be\label{defeta2}
\frac{a^*+\beta\epsilon^*}{2}<\frac{a^*+\eta\epsilon^*}{\omega_N^{1/N}}
\ee
(such a choice is possible because of~\eqref{defbeta} and because $\omega_N<2^N$, due to the assumption $N\ge 2$). Define
$$E_1:=Q_{a^*+\eta\epsilon^*}\ \hbox{ and }\ E_2:=C_{a^*+\beta\epsilon^*,r}\cup Q^x,$$
where $Q^x$ denotes the cube of measure $\lambda(Q^x):=\lambda(Q_{a^*+\eta\epsilon^*})-\lambda(C_{a^*+\beta\epsilon^*,r})>0$ and centered at $x$ with $|x|$ large enough so that $C_{a^*+\beta\epsilon^*,r}\cap Q^x=\emptyset$. Let us check that the bounded Borel sets $E_1$ and $E_2$ fulfill the desired conclusions of Theorem~\ref{th1}, provided $|x|$ is large enough (see Figure~\ref{fig:geom}).
\begin{figure}
\centering
\includegraphics[width=0.6\textwidth]{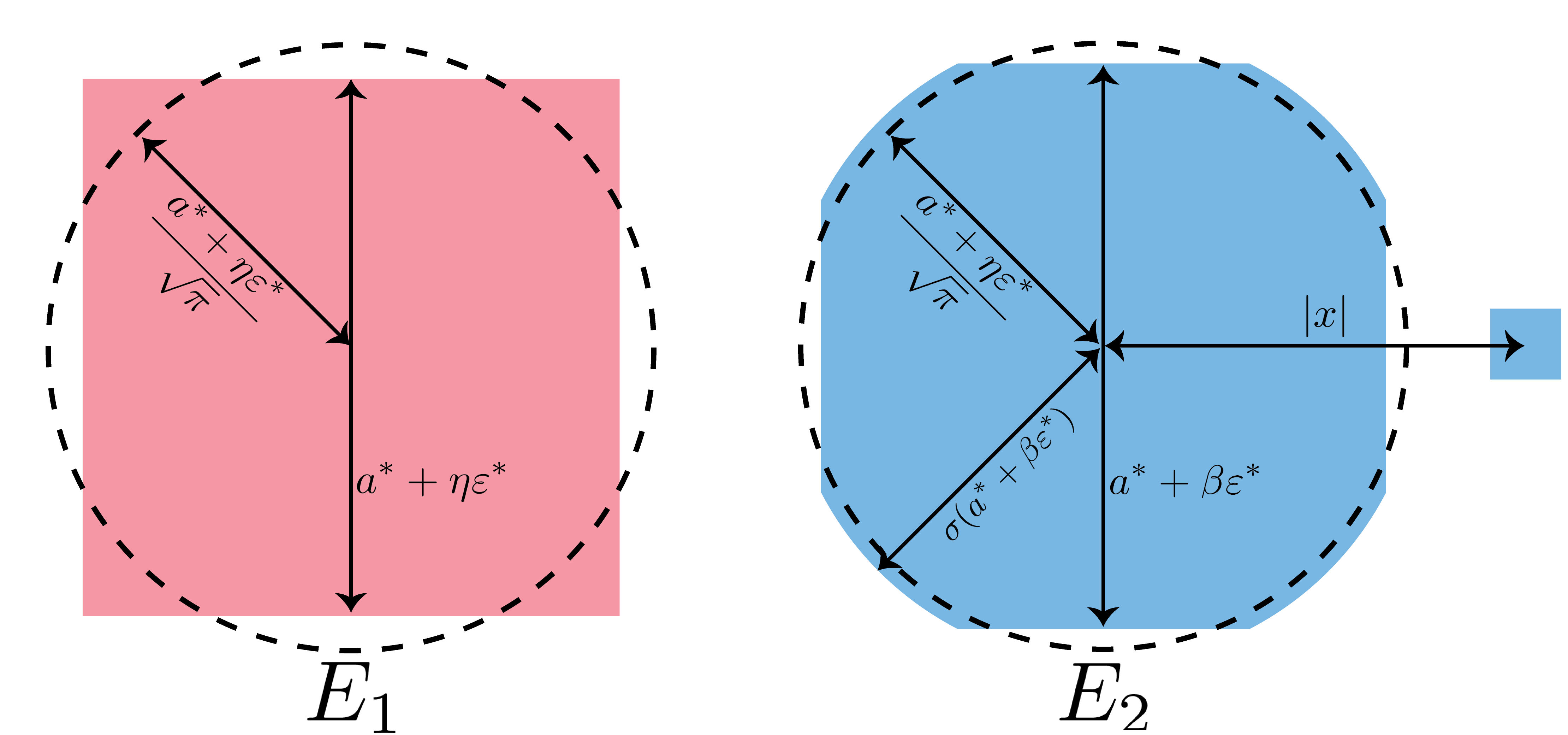}
\caption{\small{The two sets $E_1$ (in red) and $E_2$ (in blue) have the same Lebesgue measure $\lambda(E_1)=\lambda(E_2)$. Invasion occurs for~\eqref{homo} with initial condition $\mathds{1}_{E_1}$ but not with $\mathds{1}_{E_2}$. To understand why $0<\delta_1(E_2)<\delta_1(E_1)$ for all $|x|$ large enough, observe that the value of $\max_{B\in\mathcal{B},\,\lambda(B)=\lambda(E_1)}\lambda(E_1\cap B)$ corresponds in dimension $N=2$ to the measure $\lambda(E_1\cap\tilde{B})$ of the intersection between $E_1$ and the disk $\tilde{B}$ inside the dashed circle. For $|x|$ large enough, the value of $\max_{B\in\mathcal{B},\,\lambda(B)=\lambda(E_2)}\lambda(E_2\cap B)$ simply corresponds to the measure of $\lambda(E_2\cap \tilde{B})$, and is higher than $\lambda(E_1\cap \tilde{B})=\max_{B\in\mathcal{B},\,\lambda(B)=\lambda(E_1)}\lambda(E_1\cap B)$, hence $\delta_1(E_2)<\delta_1(E_1)$. The details are provided below.}}
\label{fig:geom}
\end{figure}

First of all, by definition of $a^*$ in~\eqref{defa*}, invasion occurs for the solution of~\eqref{homo} with the initial condition $\mathds{1}_{E_1}$, that is,~$E_1\in\mathcal{I}$. By definition of $r^*(\varepsilon^*)$, extinction occurs for the solution of~\eqref{homo} with the initial condition $\mathds{1}_{C_{a^*+\beta\varepsilon^*,r}}$, as
$$r=\sigma\,(a^*+\beta\varepsilon^*)<\sigma\,(a^*+\varepsilon^*)=r^*(\varepsilon^*)\le r^*(\beta\epsilon^*),$$
that is, $C_{a^*+\beta\varepsilon^*,r}\in\mathcal{E}$. Additionally,
\be\label{lambdaQx}
0<\lambda(Q^x)=\lambda(Q_{a^*+\eta\epsilon^*})-\lambda(C_{a^*+\beta\epsilon^*,r})<\frac{\omega_N{a^*}^N}{2^N}<{a^*}^N
\ee
by~\eqref{defeta1} and $\omega_N<2^N$. As a consequence, by definition of $a^*$ in~\eqref{defa*} and the invariance of~\eqref{homo} by translation, extinction occurs for the solutions of~\eqref{homo} with the initial conditions~$\mathds{1}_{Q^x}$ (that is, $Q^x\in\mathcal{E}$) for all $|x|$ large enough. As for $H_n$ in~\eqref{HnE} in the first proof of Theorem~\ref{th1}, for the Cauchy problem~\eqref{homo} with the initial condition $\mathds{1}_{E_2}$, extinction occurs (that is,~$E_2\in\mathcal{E}$) for all $|x|$ large enough.

Now, observe that $\lambda(E_1)=\lambda(E_2)>0$ for all $|x|$ large enough, and
$$\delta_1(E_1)=1-\frac{\lambda(C_{a^*+\eta\varepsilon^*,R_E})}{\lambda(E_1)},$$
with
\be\label{defRE2}
\frac{a^*+\beta\epsilon^*}{2}<R_E:=\frac{a^*+\eta\varepsilon^*}{\omega_N^{1/N}}<\frac{a^*+\beta\epsilon^*}{\omega_N^{1/N}},
\ee
where the first inequality above holds because of~\eqref{defeta2}. Lastly, for $|x|$ large enough such that $C_{a^*+\beta\epsilon^*,r}\cap Q^x=\emptyset$, one has
\be\label{lambdaQx2}
\lambda(Q^x)<\frac{\omega_N{a^*}^N}{2^N}<\frac{\omega_N(a^*+\beta\epsilon^*)^N}{2^N}
\ee
by~\eqref{lambdaQx}. On the other hand, $R_E<r$ by~\eqref{defr} and~\eqref{defRE2}, and then
\be\label{rRE}
\lambda(C_{a^*+\beta\epsilon^*,r}\cap B_{R_E})=\lambda(C_{a^*+\beta\epsilon^*,R_E})>\frac{\omega_N(a^*+\beta\epsilon^*)^N}{2^N}
\ee
since $R_E>(a^*+\beta\epsilon^*)/2$ by~\eqref{defRE2}. For every $|x|$ large enough, and for every $y\in\R^N$, one has either $C_{a^*+\beta\epsilon^*,r}\cap B_{R_E}(y)=\emptyset$ or $Q^x\cap B_{R_E}(y)=\emptyset$. Hence
$$\max_{y\in\R^N}\lambda\big(E_2\cap B_{R_E}(y)\big)=\max\Big(\max_{y\in\R^N}\lambda\big(C_{a^*+\beta\epsilon^*,r}\cap B_{R_E}(y)\big),\max_{y\in\R^N}\lambda\big(Q^x\cap B_{R_E}(y)\big)\Big)$$
for all $|x|$ large enough. Since $\lambda(Q^x)<\omega_N(a^*+\beta\epsilon^*)^N/2^N<\lambda(C_{a^*+\beta\epsilon^*,r}\cap B_{R_E})$ for all $|x|$ large enough by~\eqref{lambdaQx2}-\eqref{rRE}, one infers that
$$\max_{y\in\R^N}\lambda\big(E_2\cap B_{R_E}(y)\big)=\max_{y\in\R^N}\lambda\big(C_{a^*+\beta\epsilon^*,r}\cap B_{R_E}(y)\big)$$
and then $\max_{y\in\R^N}\lambda\big(E_2\cap B_{R_E}(y)\big)=\lambda(C_{a^*+\beta\epsilon^*,r}\cap B_{R_E})=\lambda(C_{a^*+\beta\epsilon^*,R_E})$. Therefore,
$$\delta_1(E_2)=1-\frac{\lambda(C_{a^*+\beta\varepsilon^*,R_E})}{\lambda(E_1)}>0$$
for all $|x|$ large enough (notice that $\delta_1(E_2)>0$ because $E_2$ is not a ball up to a negligible set). Finally, since $a^*+\eta\epsilon^*<a^*+\beta\epsilon^*<2R_E$ by~\eqref{defRE2}, it follows that~$\lambda(C_{a^*+\eta\varepsilon^*,R_E})<\lambda(C_{a^*+\beta\varepsilon^*,R_E})$, hence
$$0<\delta_1(E_2)<\delta_1(E_1),$$
for all $|x|$ large enough. The alternate proof of Theorem~\ref{th1} in dimension $N\ge2$ is thereby complete. 
\end{proof}

%%%%%%%%%%%%%%%%%%%%%%%%%%%%%%%%%%%%%%%%%%%%%%%%%%%%%%%%%%
%%%%%%%%%%%%%%%%%%%%%%%%%%%%%%%%%%%%%%%%%%%%%%%%%%%%%%%%%%
% 
% \section{Wasserstein distance}
% 
% We define an index $\delta_W$ based on the first order Wasserstein distance.  For any measurable set  $X\subset \R^n$, we define the corresponding probability measure $\mu_{X}(y):=\mathds{1}_{y \in X}$. Using the dual formulation of the first order Wasserstein distance, we define
% \begin{equation}
%  W_1(\mu_{X},\mu_{E}):=\sup_{f\in \mathcal{F}}\left(\int_{X}f(y)\, dy -  \int_{E}f(y)\, dy \right),
% \end{equation}
% with $\mathcal{F}$ the set of 1-Lipschitz functions from $\R^n$ to $\R$. We then define the fragmentation index $\delta_W$ as follows:
% \begin{equation}
%     \delta_W(E)=\inf_{B\in\mathcal{B},\,\lambda(B)=\lambda(E)}1-\exp( -\kappa \, W_1(\mu_{B},\mu_{E})/R_E),
% \end{equation}
% with $\kappa>0$ a fixed constant. Note that $\delta_W$ is invariant by contraction or dilatation:  $\delta_W(\mu \, E)=\delta_W(E)$ for every bounded non-negligible Borel set $E$ and for every $\mu>0$.
% 
% As mentioned earlier, the example of Fig.~\ref{fig:geom} should not work with the index $\delta_W$. Contrarily to the index $\delta_1$, which does not take into account the distance between the two components of $E_2$ ***

%%%%%%%%%%%%%%%%%%%%%%%%%%%%%%%%%%%%%%%%%%%%%%%%%%%%%%%%%%
%%%%%%%%%%%%%%%%%%%%%%%%%%%%%%%%%%%%%%%%%%%%%%%%%%%%%%%%%%

\end{document}